\documentclass{amsart}
\usepackage{graphicx} 
\usepackage{amsmath,amsfonts,amsthm,amssymb,tikz-cd,hyperref, mathrsfs, tikz}
\usepackage[backend=bibtex,style=alphabetic,sorting=nyt]{biblatex}
\usepackage{cleveref}
\crefname{ineq}{Inequality}{inequalities}
\usepackage[shortlabels,inline]{enumitem}

\newtheorem{theorem}{Theorem}[section]
\newtheorem{lemma}[theorem]{Lemma}
\newtheorem{proposition}[theorem]{Proposition}
\newtheorem{corollary}[theorem]{Corollary}
\newtheorem{thmdef}[theorem]{Theorem-Definition}

\theoremstyle{definition}
\newtheorem{example}[theorem]{Example}
\newtheorem{definition}[theorem]{Definition}

\newtheorem{remarklabeled}[theorem]{Remark}

\numberwithin{equation}{section}
\DeclareMathOperator{\GL}{GL}

\DeclareMathOperator{\SL}{SL}

\DeclareMathOperator{\id}{id}

\DeclareMathOperator{\Supp}{Supp}

\DeclareMathOperator{\Ess}{Ess}

\DeclareMathOperator{\reg}{reg}

\DeclareMathOperator{\Spec}{Spec}
\DeclareMathOperator{\im}{im}
\DeclareMathOperator{\Hom}{Hom}

\DeclareMathOperator{\Irr}{Irr}
\DeclareMathOperator{\Stab}{Stab}

\newcommand{\NamiWeyl}{W^{\mathrm{N}}}
\newcommand{\inverse}{^{-1}}
\newcommand{\gitquo}{/\!/}

\newcommand{\RR}{\mathbb{R}}
\newcommand{\M}{\mathcal{M}}
\newcommand{\F}{\mathcal{F}}
\newcommand{\g}{\mathfrak{g}}
\newcommand{\C}{\mathbb{C}}
\newcommand{\ZZ}{\mathbb{Z}}

\newcommand{\mf}{\mathfrak}
\newcommand{\sheafO}{\mathcal{O}}
\newcommand{\cartanh}{\mathfrak{h}}

\newcommand{\mc}{\mathcal}

\newcommand{\gl}{\mathfrak{gl}}

\renewcommand{\>}{\rangle}
\renewcommand{\L}{\mathcal{L}}

\title{Maffei's action and symplectic Springer action for quiver varieties}
\author{Yaochen Wu}

\begin{document}
	\begin{abstract}
		We examine the relationship between the actions of two Weyl groups on the cohomology of a smooth quiver variety: the Maffei's action of the Weyl group associated to the quiver, defined in \cite{maffei2002remark}, and the symplectic Springer action of the Namikawa-Weyl group of the affine quiver variety, defined in \cite{mcgerty2019springer}. 
		We show there is a natural map from the former group to the latter, which is an embedding in favorable situations, and this map intertwines their actions on the cohomology. This answers a question raised in \cite{mcgerty2019springer}. 
	\end{abstract}
	\maketitle

	\tableofcontents
	\section{Introduction}
	\subsection{Quiver varieties}\label{subsection:quiver varieties intro}
	We recall the construction of quiver varieties following \cite{nakajima1998quiver}. 
	Let $Q$ be a quiver without edge loops. Let $Q_0,Q_1$ be its set of vertices and arrows respectively. 
	For $i\neq j\in Q_0$, denote by $n_{ij}$ the number of arrows between $i$ and $j$, regardless of the orientation. 
	Set $c_{ij} = -n_{ij}$ for $i\neq j$, and $c_{ii} = 2$.  
	Let $\alpha_i$ denote the simple real root attached to the vertex $i$. Denote $\cartanh_Q$ to be the complex vector space spanned by the basis $ \{\alpha_i|i\in Q_0\}$. Define the Tits form on $\cartanh_Q$ by 
	\begin{equation}
		(\alpha_i,\alpha_j) = c_{ij}. 
	\end{equation}
	The matrix $(c_{ij})$ is a symmetric Cartan matrix. The \textit{roots} of $Q$, real or imaginary, are defined to be the roots of the corresponding Kac-Moody algebra. If $\alpha$ is a real root then $(\alpha,\alpha)=2$. If $\alpha$ is an imaginary root, we say $\alpha$ is \textit{isotropic} (respectively, \textit{non-isotropic}) if $(\alpha,\alpha) = 0$ (respectively, if $(\alpha,\alpha) < 0$).  

	Let $v = (v_i)_{i\in Q_0} \in \ZZ_{\ge 0}^{Q_0}$ be a dimension vector, and $w = (w_i)_{i\in Q_0}$ be a framing vector. 
	A \textit{framed representation} of $Q$ with dimension $v$ and framing $w$ consists of the following data: 
	\begin{enumerate}
		\item To each vertex $i\in Q_0$ we assign a vector space $V_i$ that has dimension $v_i$, and a morphism $p_i\in Hom (W_i,V_i)$, where $W_i$ is a fixed $w_i$-dimensional vector space. 
		\item To each arrow $a\in Q_1$ we assign an element $x_a\in \Hom (V_{t(a)}, V_{h(a)})$. 
	\end{enumerate}
	
	Define 
	\begin{equation}{\label{definition of R}}
		R = R(Q,v,w) := \bigoplus_{\substack{a\in Q_1}} \Hom (V_{t(a)}, V_{h(a)}) 
		\oplus \bigoplus_{i\in Q_0} \Hom (W_i,V_i).
	\end{equation}
	This is the space of representations of $Q$ with dimension vector $v$ and framing $w$. 
	We identify $T^*R$ with the space of representations of $\overline{Q}$, the \textit{double quiver} of $Q$, i.e. for each $a \in Q_1$ we add an arrow $\overline{a}$ so that $h(\overline{a}) = t(a), t(\overline{a}) = h(a)$, and for each $i\in Q_0$ we add a coframing in $\Hom(V_i,W_i)$. 
	For an element $r\in T^*R$, we write 
	\begin{equation}\label{eq:expression of a representation}
		r = (x_a,y_a,p_i,q_i)_{a\in Q_1,i\in Q_0},
	\end{equation}
	where $x_a\in \Hom (V_{t(a)}, V_{h(a)}) , y_a\in \Hom (V_{h(a)}, V_{t(a)}), p_i\in \Hom (W_i,V_i)$ and $ q_i\in \Hom (V_i,W_i)$. 
	
	Write $G = GL(v):= \prod_{i\in Q_0} \GL(V_i)$. This group acts naturally on $R$, and therefore we have an induced Hamiltonian $G$-action on $T^*R$. Let $\mu: T^*R\to \g^*$ be the standard moment map. 
	
	We can identify $(\g^*)^G \cong \C^{Q_0}$, and write $\mathfrak{p} := \C^{Q_0}$. 	
	Let $\lambda \in \mathfrak{p}$. Let $\theta$ be a character of $G$; we can view $\theta$ as an element of $\ZZ^{Q_0}$. 
	We view $\theta$ and $\lambda$ as elements of $\cartanh_Q^*$. 
	By definition, the Nakajima quiver variety $\M_\lambda^\theta(Q,v,w)$ is the GIT quotient
	\begin{equation}
		\M_\lambda^\theta(Q,v,w) := \mu\inverse(\lambda)^{\theta -ss} \gitquo G.
	\end{equation} 
	When the quiver $Q$ is clear from the context, we will write $\M_\lambda^\theta(v,w)$ for the quiver variety. Note that if $\theta = 0$, then $\mu\inverse(\lambda)^{\theta-ss} = \mu\inverse(\lambda)$ and $\M_\lambda^0(v,w)$ is affine. 
	We also define
	\begin{equation}
		\M_\mathfrak{p}^\theta(v,w) := \mu\inverse(\mathfrak{p})^{\theta -ss}\gitquo G.
	\end{equation}
	There is a natural map $\M_\mathfrak{p}^\theta(v,w) \to \mathfrak{p}$ induced from $\mu$, and $\M_\lambda(v,w)$ is the fiber of this map over $\lambda\in\mathfrak{p}$. 
	
	We introduce the following notations. 
	\begin{enumerate}
		\item We denote by $\rho$ the natural projective morphism $\M_\lambda^\theta(v,w) \to \M_\lambda^0(v,w)$.
		\item If $r\in \mu\inverse(\lambda)^{\theta-ss}$, then we denote its image in $\M_\lambda^\theta(v,w)$ by $[r]^\theta$. 
		\item If $r\in \mu\inverse(\lambda)$, then we denote its semisimplification as a $\overline{Q}$-representation by $r^{ss}$. 
	\end{enumerate}
	
	We say the pair $(\theta,\lambda)$ is \textit{generic} if for any $v' \in \ZZ_{\ge 0}^{Q_0}$, $v' \le v$, we have $(v'\cdot \theta, v'\cdot \lambda) \neq (0,0)$. We say $\theta$ (resp. $\lambda$) is generic if $(\theta,0)$ (resp. $(0,\lambda)$) is generic. 
	
	We collect a few known results about quiver varieties. 
	\begin{enumerate}
		\item If $r \in \mu\inverse(\lambda)^{\theta-ss}$ then $\rho([r]^\theta) = [r^{ss}]^0$. See \cite[Theorem 2.10, Section 10.1]{kirillov2016quiver}. 
		\item If $(\theta,\lambda)$ is generic, then  the $G$-action on $\mu\inverse(\lambda)^{\theta-ss}$ is free and $\M_\lambda^\theta(v,w)$ is smooth.  See \cite[Theorem 2.8]{nakajima1994instantons} and its proof, and \cite[Remark 3.16]{nakajima1998quiver}. 
		In particular, if $\lambda$ is generic, then $\mu\inverse(\lambda) = \mu\inverse(\lambda)^{\theta -ss}$ for any $\theta$, and the projective map $\rho:\M_\lambda^\theta(v,w) \to \M_0^\theta(v,w)$ is an isomorphism. 
		\item If $\theta$ is generic, then the affinization $\Spec \C[\M_0^\theta(v,w)]$ is independent of the choice of $\theta$. We denote it by $\M_0(v,w)$. See \cite[Corollary 2.4]{bezrukavnikov2021etingof}. 
		\item If the moment map $\mu$ is flat, then $\M_0(v,w)\cong \M_0^0(v,w)$. In this case, $\M_0^\theta(v,w)\to \M_0^0(v,w)$ is a conical symplectic resolution of singularities. See \cite[Proposition 2.3, Proposition 2.5]{bezrukavnikov2021etingof}. 
	\end{enumerate}	
	
	It will be useful to identify  $\M_\lambda^\theta(v,w)$ with a quiver variety without framing by the method in \cite[Section 1, Remarks]{crawley2001geometry}. Namely, we define a new quiver ${Q^\infty}$ as follows. We set ${Q^\infty_0} = Q_0 \coprod \{\infty\}$. The arrows between the vertices in $Q_0$ are the same as those in $Q_1$. In addition, for each $i\in Q_0$, there are $w_i$ arrows from the vertex $i$ to the vertex $\infty$. 
	Denote the new simple root associated to the vertex $\infty$ by $\alpha_\infty$. Define the extended dimension vector $\tilde{v}\in \ZZ_{\ge 0}^{{Q^\infty_0}}$ by $\tilde{v}_\infty = 1$ and $\tilde{v}_i = v_i$ for $i\in Q_0$. 
	It is clear that $R(Q,v,w) = R({Q^\infty},\tilde{v})$, so  $\M^\theta_\lambda(Q,v,w) = \M_{\tilde{\lambda}}^{\tilde{\theta}}({Q^\infty},\tilde{v})$, 
	where $\tilde{\lambda}_\infty = -v\cdot \lambda$, $\tilde{\lambda}_i = \lambda_i$ for $i\in Q_0$; $\tilde{\theta}_\infty = -\theta \cdot v, \tilde{\theta}_i = \theta_i$ for $i \in Q_0$. To simplify notations, we will write $\M_\lambda^\theta(\tilde v)$ for $\M_{\tilde{\lambda}}^{\tilde{\theta}}({Q^\infty},\tilde{v})$ when there is no danger of confusion. 
	
	\subsection{Quiver Weyl group}
	We define the following Coxeter group associated to $Q$. 
	\begin{definition}
		The \textit{quiver Weyl group} $W_Q$ of $Q$ is generated by $\{s_i|i\in Q_0\}$, subject to the relations: 
		\begin{enumerate}
			\item $s_i^2 =1$;
			\item $s_is_j=s_js_i$ if $n_{ij} = 0$; 
			\item $(s_is_j)^3 = 1$ if $n_{ij} = 1$. 
		\end{enumerate}
	\end{definition}
	We remark that we impose no relations between $s_i$ and $s_j$ when $n_{ij} \ge 2$. 
	The group $W_Q$ is the same as the Weyl group associated to the Kac-Moody algebra defined by the symmetric Cartan matrix $(c_{ij})$. 
	
	Similarly we can define $W_{Q^\infty}$. We view $W_Q$ as a subgroup of $W_{Q^\infty}$. 
	
	There is a natural representation of $W_{Q}$ in $\cartanh_Q$ given by 
	\begin{equation}\label{eq: def of W action on Q}
		s_i(\alpha_j) = \alpha_j -c_{ij}\alpha_i .
	\end{equation}
	The dual representation is described as follows. If $\chi = \sum \chi_i\alpha_i^* \in \cartanh_Q^*$ where $\alpha_i^*$ is the dual basis to $\alpha_i$, then 
	\begin{equation}\label{eq: formula of si act on chi}
		s_i^*(\chi) = \sum_{j\in Q_0} (\chi_j-c_{ij}\chi_i)\alpha_j^*.
	\end{equation}
	Similarly, we can define the natural representations of $W_{Q^\infty}$ in $\cartanh_{Q^\infty}$ and $\cartanh_{Q^\infty}^*$. We have a representation of $W_Q$ on $\cartanh_{Q^\infty}$ by the same formula as \eqref{eq: def of W action on Q}, and allow $j=\infty$.
	\begin{definition}
		We denote the stabilizer of $\tilde{v}$ in $W_Q$ by $W(v,w)$. 
	\end{definition}
	
	\subsection{Poisson deformations and Namikawa-Weyl groups}
	We recall the definition of the Namikawa-Weyl group associated to a conical symplectic singularity.  
	\begin{definition}[{\cite{beauville2000symplectic}}]\label{conical symplectic singularities}
		Let $X$ be a normal affine Poisson variety.
		\begin{enumerate}
			\item We say $X$ \textit{has symplectic singularities} if the Poisson structure on its smooth locus $X^{\reg}$ is nondegenerate, thus inducing a symplectic form $\omega$, and there is a projective resolution of singularities $\rho: Y \to X$ such that $\rho^* \omega$ extends to a regular 2-form on $Y$. Note that $\rho^*\omega$ need not be nondegenerate. 
			\item We say $X$ is \textit{conical} if there is a $\C^*$-action on $X$ that contracts $X$ to a point, and $\omega$ has positive weight under this action.
			\item We say $\rho$ is a \textit{symplectic resolution of singularities} if $\rho^* \omega$ extends to a symplectic form on $Y$; we say $\rho$ is \textit{conical} if $Y$ is equipped with a $\C^*$-action so that $\rho$ is $\C^*$-equivariant. 
		\end{enumerate}
	\end{definition}
	\begin{definition}\label{definition graded poisson deformation}
		Let $X$ be a normal Poisson variety equipped with a $\C^*$-action, and the Poisson bracket has degree $-d$ for some $d\in \ZZ_{>0}$. A \textit{graded Poisson deformation} of $X$ is the data $(\mc{X}, B, j)$, where: 
		\begin{enumerate}
			\item $B = \bigoplus_{i\ge 0} B_i$ is a finitely generated graded $\C$-algebra, such that the degree 0 piece $B_0 = \C$. 
			\item $\mc{X}$ is a Poisson $B$-variety equipped with a $\C^*$-action, the Poisson structure on $\sheafO_\mc{X}$ is $B$-linear, and the structure morphism $m_X: \mc{X}\to \Spec(B)$ is $\C^*$-equivariant and flat. 
			\item $j: X\xrightarrow{\sim} m_X\inverse(0)$ is a $\C^*$-equivariant isomorphism of Poisson varieties, where $0\in \Spec(B)$ corresponds to the maximal ideal $\bigoplus_{i>0} B_i$. 
		\end{enumerate}
	\end{definition}
	Let $X$ be as in \Cref{definition graded poisson deformation} and $(\mc{X},B, j), (\mc{X}',B', j')$ be two graded Poisson deformations. 
	A \textit{morphism} of graded Poisson deformations from $(\mc{X},B , j)$ to $ (\mc{X}',B', j')$ consists of $\C^*$-equivariant morphisms $\Phi:\mc{X}\to \mc{X}'$ and  $f: \Spec(B)\to \Spec(B')$, such that the following diagram is Cartesian
	\[
	\begin{tikzcd}
		\mc{X}\arrow[r,"\Phi"] \arrow[d,"m_X"] & \mc{X}' \arrow[d,"m_X'"]\\
		\Spec(B)\arrow[r,"f"] &  \Spec(B')
	\end{tikzcd}
	\]
	and moreover, $j' = \Phi|_{\pi\inverse(0)} \circ j$. 
	
	We say a graded Poisson deformation $(\mc{X} ,B, j)$ is \textit{universal} if for any graded Poisson deformation $(\mc{X}',B', j')$, there is a unique morphism of graded Poisson deformations from $(\mc{X}',B', j')$ to $(\mc{X} ,B, j)$.

	Let $X$ be a conical symplectic variety, $\rho: Y\to X$ be a conical symplectic resolution of singularities. 
	\begin{theorem}{\label{Namikawa diagram of universal deformation }} {\cite[Theorem 5.5]{namikawa2011poisson}}
		There is a commutative diagram 
		\begin{equation}{\label{diagram of univ poisson defor of symplectic resol}}
			\begin{tikzcd}
				\mc{Y} \arrow[r] \arrow[d,"m_Y"] & \mc{X} \arrow[d,"m_X"]\\
				B_Y  \arrow[r,"q"]& B_X
			\end{tikzcd}
		\end{equation}
		where $B_Y = H^2(Y,\C) $, and  $m_X$, $m_Y$ are universal graded Poisson deformations of $X$ and $Y$ respectively, with $m_X\inverse(0) = X, m_Y\inverse(0) = Y$.

	\end{theorem}
	We write $\cartanh_X: = H^2(Y,\C)$. By \cite[Proposition 2.18]{braden2016conical}, $\cartanh_X$ is independent of the choice of the  symplectic resolution $Y$. We call it the \textit{Namikawa-Cartan space} of $X$. 
	
	By {\cite[Theorem 2.3]{kaledin2006symplectic}}, $X$ has finitely many symplectic leaves. 
	Let $\mc{L}_1,...,\mc{L}_n$ be the codimension 2 symplectic leaves of $X$. 
	The formal slice $S_i$ to $\mc{L}_i$ is a type $ADE$ Kleinian singularity; let $\widehat{W}_i$ and $\hat{\cartanh}_i^*$ be the corresponding Weyl group and root space, respectively. The fundamental group $\pi_1(\mc{L}_i)$ acts on $\widehat{W}_i$ and $\hat{\cartanh}_i^*$ by Dynkin diagram automorphisms. 
	Define $W_i := (\widehat{W}_i)^{\pi_1(\mc{L}_i)}$ and $\cartanh_i := (\hat{\cartanh}_i^*)^{\pi_1(\mc{L}_i)}$. They are the Weyl groups and Cartan spaces corresponding to the folded Dynkin diagram; the foldings are given by the action of $\pi_1(\mc{L}_i)$. 
	
	The Namikawa-Cartan space $\cartanh_X$ has the following decomposition. 
	\begin{theorem}[{\cite[Lemma 2.8]{losev2022deformations}}]{\label{Ivan's decomposition of cartan space}} There is a vector space isomorphism
		\begin{equation}\label{eq: decomposition of namikawacartan space}
			H^2(Y,\C) = H^2(X^{\reg},\C) \oplus \bigoplus_{i=1}^n \mf{h}_i
		\end{equation} 
		where $\L_1,\cdots,\L_n$ are the codimension 2 leaves of $X$. 
	\end{theorem} 

	\begin{remarklabeled}\label{remark: projection from H2 to cartanh i}
		We describe the projection $H^2(Y,\C) \twoheadrightarrow \cartanh_i$, corresponding to the codimension 2 leaf $\L_i$. Let $x\in \L_i$. Then, by \cite[Section 4.1]{namikawa2011poisson}, there is an analytic neighbourhood $U$ of $x$ in $X$ such that
		\begin{enumerate}
			\item there is a Poisson isomorphism $U \cong S_i \times \Delta^{\dim X -2}$, where $ \Delta^{\dim X -2}$ is the complex polydisc of dimension $\dim X -2$ and $S_i$ is an analytic neighbourhood of the conical point of the corresponding Kleinian singularity;
			\item $\rho\inverse(U) \cong \tilde{S}_i \times \Delta^{\dim X -2} $, where $\tilde{S}_i$ is the preimage of $S_i$ in the minimal resolution of the Kleinian singularity. 
		\end{enumerate} 
		Take $\alpha\in H^2(Y,\C)$, restrict it to $\rho\inverse(U)$, and we get an element $\alpha_i \in H^2(\tilde{S_i},\C)$. The latter is isomorphic to the root space $\hat{\cartanh}_i^*$. By \cite[Proposition 4.2]{namikawa2011poisson}, $\alpha_i$ is invariant under the $\pi_1(\L_i)$ action, i.e. $\alpha_i \in \cartanh_i$. 
		By definition, the projection $H^2(Y,\C) \twoheadrightarrow \cartanh_i$ maps $\alpha$ to $\alpha_i$.
	\end{remarklabeled}
	\begin{definition}
		The direct product
		\[\NamiWeyl = \prod_{i = 1}^n W_i,\]
		where $\L_1,\cdots,\L_n$ are the codimension 2 symplectic leaves of $X$, is called the \textit{Namikawa-Weyl group} of $X$. 
	\end{definition}
	\begin{remarklabeled}\label{remark natural action NW on namikawacartan}
		There is a natural $\NamiWeyl$-action on $H^2(Y,\C)$, namely each $W_i$ acts as the Weyl group on the $\cartanh_i$ component in \eqref{eq: decomposition of namikawacartan space}, and trivially on all other components. 
	\end{remarklabeled}
	The Namikawa-Weyl group is important for the following reason. 
	\begin{theorem}[{\cite[Theorem 1.1]{namikawa2010poisson}}]{\label{Namikawa Weyl group}}
		The map $q$ in \eqref{diagram of univ poisson defor of symplectic resol} is the quotient map of the natural $\NamiWeyl$-action on $H^2(Y,\C)$. 
	\end{theorem}
	When $\M_0^\theta(v,w)$ is nonempty, its affinization $\M_0(v,w)$ is isomorphic to some $\M_0^0(Q',v',w')$ for a subquiver $Q'\subset Q$ (with the possibility $Q' = Q$), and some nonzero dimension vector and framing vector $v',w'$ associated to $Q'$. This is proved using Maffei's isomorphisms in \cite[Proposition 2.14]{wu2023namikawa}. Therefore, by \cite[Theorem 1.5]{bellamy2021symplectic}, $\M_0(v,w)$ has symplectic singularities and admits a conical symplectic resolution. 
	\begin{definition}
		We denote the Namikawa-Weyl group of $\M_0(v,w)$ by $\NamiWeyl(v,w)$.
	\end{definition}
	In the setup of \Cref{Namikawa diagram of universal deformation }, let $X = \M^0_0(v,w)$, $Y = \M_0^\theta(v,w)$, and assume $Y\to X$ is a conical symplectic resolution of singularities (for example, when $\mu$ is flat). 
	The natural map $\M_\mathfrak{p}^\theta(v,w) \to \mathfrak{p}$
	is a conical Poisson deformation of $X$. 
	By the universality of the deformation $\mathcal{Y} \to B_Y$, there is a unique $\C^*$-equivariant map $\kappa: \mathfrak{p}\to B_Y = H^2(\M_0^\theta(v,w),\C) $, such that the following diagram is Cartesian:
	\begin{equation}\label{diag: kappa cartesian}
		\begin{tikzcd}
			\M_{\mathfrak{p}}^\theta(v,w) \arrow[r] \arrow[d,"\mu"] & \mathcal{Y} \arrow[d,"m_Y"] \\ 
			\mathfrak{p} \arrow[r,"\kappa"] & B_Y
		\end{tikzcd}
	\end{equation}
We will give more information about $\kappa$ in \Cref{theorem: kappa is chern class and surjective}. 
	\subsection{Goal and structure of the paper}
	Suppose $Y\to X$ is a conical symplectic resolution, and $\NamiWeyl$ is the Namikawa-Weyl group of $X$. 
	In \cite{mcgerty2019springer}, the authors construct the \textit{symplectic Springer} action of $\NamiWeyl$ on $H^*(Y,\C)$. We will recall this construction in \Cref{subsection: The symplectic Springer action}. In particular, there is a symplectic Springer action of $\NamiWeyl(v,w)$ on $H^*(\M_0^\theta(v,w),\C)$. 	
	On the other hand, \cite{maffei2002remark} constructed an action of $W(v,w)$ on $H^*(\M^\theta_0(v,w),\C)$, which we call  \textit{Maffei's action}, and we will recall it in \Cref{section:W action on H2}. 
	
	When $Q$ is a Dynkin quiver of type $ADE$ and $(\tilde{v},\alpha_i) \le 0$ for all $i\in Q_0$, $W(v,w) = \NamiWeyl(v,w)$ by \cite[Theorem 5.4]{mcgerty2019springer}. 
	In \cite[Remark 5.9]{mcgerty2019springer}, the authors raise questions on the relations between $W(v,w)$ and $\NamiWeyl(v,w)$, and the Maffei and symplectic Springer actions on $H^*(\M_0^\theta(v,w),\C)$ for a general quiver. We aim to answer these questions. 
	
	For any quiver $Q$ without edge loops denote the \textit{fundamental region} of $Q$ to be 
	\begin{equation}\label{eq: def of fundamental domain}
		\F_Q = \{\alpha\in\cartanh_Q|(\alpha,\alpha_i)\le 0 \ \forall i\in Q_0\}. 
	\end{equation}
	In \cite{crawley2001geometry} the author defined an important set $\Sigma_0$ of positive roots of a quiver, which we recall in \Cref{CB simple dimension criteria}. The main result of this paper is the following. 
	\begin{theorem}\label{theorem:W vs NW}
		Let $Q$ be a quiver without edge loops, $\theta$ be generic, $\tilde{v}\in \F_{Q^\infty}$, and $\M_0^\theta(v,w)$ be nonempty. 
		Then there is a direct product decomposition $W(v,w) = W(v,w)_{\Irr} \times W(v,w)_{\Ess}$, satisfying the following properties.
		\begin{enumerate}
			\item The subgroup $W(v,w)_{\Irr}$ is either trivial or a direct product of type $A$ Weyl groups. When $\tilde{v}\in\Sigma_0$, $W(v,w)_{\Irr}$ is trivial. 
			\item The restriction of Maffei's action on $H^2(\M_0^\theta(v,w),\C)$ to $W(v,w)_{\Irr}$ is trivial. 
			\item There is an embedding $\iota: W(v,w)_{\Ess} \hookrightarrow \NamiWeyl(v,w)$, and the natural action of $\NamiWeyl(v,w)$ on $H^2(\M_0^\theta(v,w),\C)$ and the Maffei's action of $W(v,w)_{\Ess}$ on $H^2(\M_0^\theta(v,w),\C)$ are intertwined under $\iota$. 
			\item The map $1\times \iota: W(v,w)_{\Irr} \times W(v,w)_{\Ess} \to \NamiWeyl(v,w), (\sigma_1,\sigma_2) \mapsto \iota(\sigma_2),$ intertwines the Maffei's action of $W(v,w)$ and the symplectic Springer action of $\NamiWeyl(v,w)$ on $H^*(\M_0^\theta(v,w),\C)$.  
		\end{enumerate}
	\end{theorem}
	In \Cref{section:Preliminaries on weyl groups} we present some preliminary results on $W(v,w)$ and $\NamiWeyl(v,w)$, and examine the canonical decomposition of Crawley-Boevey under the assumption $\tilde{v}\in\F_{Q^\infty}$. 
	In \Cref{subsection: The symplectic Springer action}, we recall the construction of the symplectic Springer action. 
	In \Cref{section:W action on H2}, we recall the notion of tautological line bundles, and explicitly calculate the Maffei's action of $W(v,w)$ on $H^2(\M_0^\theta(v,w),\C)$. This will prove parts (2) and (3) of \Cref{theorem:W vs NW}, and we will prove (4) in \Cref{wrap up part 4}. 
	\subsection*{Acknowledgments}
	I am grateful to Ivan Losev for fruitful discussions and numerous remarks that improve the exposition of this paper, in particular for pointing out that (4) of \Cref{theorem:W vs NW} is not hard once (1)-(3) are proven. I thank Alberto San Miguel Malaney for reading a first version of this paper and providing many useful comments. I also thank Pablo Boixeda Alvarez, Do Kien Hoang and Junliang Shen for useful discussions. 
	\section{The groups $W$ and $\NamiWeyl$}\label{section:Preliminaries on weyl groups}
	In this section we study the groups $W(v,w)$ and $\NamiWeyl(v,w)$. We define the components $W(v,w)_{\Irr}$ and $W(v,w)_{\Ess}$ and give examples that illustrate the relation between $W(v,w)$ and $\NamiWeyl(v,w)$. 
	\subsection{Generators of $W(v,w)$}
	We assume $\tilde{v} \in \F_{Q^\infty}$, the fundamental domain . 
	\begin{proposition}\label{prop:generators of Wvw}
		The group $W(v,w)$ is generated by 
		\begin{equation}\label{eq:gen of Wvw}
			\{s_i|i\in Q_0, (\alpha_i,\tilde{v}) = 0\}. 
		\end{equation}
	\end{proposition}
	\begin{proof}
		Set $P := ( \tilde v, - ) \in \cartanh_{Q}^*$, so that $\<P,\alpha_i\>\le 0$ for all $i\in Q_0$. 
		It follows from the definition of $P$ that 
		\begin{equation}\label{eq:stab of root lies in stab of functional on roots}
			\Stab_{W_{Q}}( \tilde{v}) \subseteq \Stab_{W_{Q}}(P).
		\end{equation}
		By \cite[Section 5.13]{humphreys1992reflection},  $\Stab_{W_{Q}}(P)$ is a parabolic subgroup of $W_{Q}$, 
		and it is generated by $\{s_i|i\in Q_0, \<P,\alpha_i\> = 0\}$. It is easy to check that all of these generators lie in $\Stab_{W_{Q}}( \tilde{v}) $ as well, and therefore \eqref{eq:stab of root lies in stab of functional on roots} is an equality. But $\<P,\alpha_i\> = 0$ precisely when $(\alpha_i,\tilde{v}) = 0$, so the set \eqref{eq:gen of Wvw} generates $W(v,w)$. 
	\end{proof}
	
	\subsection{Symplectic leaves of quiver varieties}
	We describe the symplectic leaves of $\M_\lambda^0(v,w)$, following \cite[Section 3.v]{nakajima1998quiver}. Since $\M_\lambda^0(v,w) \cong \M_{{\lambda}}^0(\tilde{v})$, we may assume $w=0$. The symplectic leaves are classified by \textit{representation types}.
	More precisely, 
	let $x\in\M_\lambda^0(v)$ and $r \in T^*R$ be a representative of $x$. Since the $G$-orbit of $r$ is closed, $r$ is semisimple as a $Q$-representation. 
	\begin{definition}
		Suppose $r = r_1^{\oplus n_1} \oplus ... \oplus r_k^{\oplus n_k}$, 
		where $r_i$'s are pairwise non-isomorphic irreducible $Q^\infty$-subrepresentations of $r$. Write $v^i = \dim r_i \in \ZZ_{\ge 0}^{{Q^\infty_0}}$. We say 
		\[\tau = (v^1,n_1;v^2,n_2;...;v^k,n_k)\]
		is the \textit{representation type} of $x$. 
	\end{definition}
	By \cite[Section 3.iv]{nakajima1998quiver}, a symplectic leaf of $\M_\lambda^0(\tilde{v})$ constitutes of all elements with a fixed representation type. We will often say a symplectic leaf is associated to a representation type $\tau$, and denote the leaf by $\L_\tau$. 
	
	The following results of Crawley-Boevey characterizes the dimension vectors that can appear in a representation type. 
	To state it, we define
	\begin{equation}
		p:  \ZZ_{\ge 0}^{Q_0 } \to \ZZ, v \mapsto 1-\frac{1}{2}(v,v).
	\end{equation}
	\begin{thmdef}[{\cite[Theorem 1.2]{crawley2001geometry}}]{\label{CB simple dimension criteria}}
		The following conditions are equivalent.
		\begin{enumerate}
			\item There is a simple representation of $Q^\infty$ in  $\mu\inverse(\tilde{{\lambda}})$ with dimension vector $v'$.
			\item $v'$ is a positive root of ${Q^\infty}$, $\tilde{{\lambda}}\cdot v' = 0$, and for any decomposition $v' = \beta^1+ ... + \beta^n$ where $n\ge 2$ and $\beta^i$ are positive roots of ${Q^\infty}$ such that $\beta^i\cdot \tilde{{\lambda}}=0$, we have $p(v') > \sum_{i=1}^n p(\beta^i)$. 
		\end{enumerate}
		We denote by $\Sigma_{\tilde{\lambda}}$ the set of positive roots that satisfy the above conditions. 
	\end{thmdef}
	Note that if $\tilde{v}\in \Sigma_0$, then $\tilde{v}\in \F_{Q^\infty}$, the fundamental domain of $Q^\infty$, but not vice versa. 
	
	Suppose $\tilde{v}\in \F_{Q^\infty}$. To compute the Namikawa-Weyl group of $\M_0^0(v,w)$ it is useful to find its codimension 2 leaves. The following result computes the dimension of the symplectic leaves of $\M_0^0(v,w)$. 
	\begin{theorem}[{\cite[Theorem 1.3]{crawley2001geometry}}]{\label{CB strata dimension formula}}
		Let $\tau = (v^1,n_1;...;v^k,n_k)$ be a representation type such that $\sum_{i=1}^k n_iv^i = \tilde{v}$.
		The stratum of $\M_{\tilde{{\lambda}}}^0(\tilde{v})$ associated to $\tau$ has dimension
		\begin{equation}d(\tau) = 2\sum_{i=1}^k p(v^i).\end{equation}
	\end{theorem}
	
	The following theorem is a special case of \cite[Definition 1.18, Theorem 1.20]{bellamy2021symplectic}. 
	\begin{theorem}\label{BS isotropic decomposition}
		Suppose $v\in \Sigma_0$ is an imaginary root. Then the codimension 2 leaves of $\M_0^0(v)$ corresponds to the representation types of the form
		\begin{equation}\label{eq: an isot decomp}
			\tau = (\beta_1,1;\beta_2,1;\cdots;\beta_s,1; \alpha_{i_1},m_1;\cdots;\alpha_{i_t},m_t)
		\end{equation}
		such that 
		\begin{enumerate}
			\item the $\beta_j$ are imaginary roots, and the $\alpha_{i_j}$ are pairwise distinct simple real roots;
			\item let $\underline{Q} $ be the quiver with $s+t$ vertices, with no edge loops, and $-(u_i,u_j)$ arrows between vertices $i$ and $j$, where $u_i,u_j\in \{\beta_1, \cdots \beta_s ,  \alpha_{i_1}, \cdots, \alpha_{i_t}\}$, then $\underline{Q} $ is an affine type quiver;
			\item the dimension vector $(1,1,\cdots,1;m_1,\cdots,m_t)$ (with $s$ copies of 1's) equals to the minimal imaginary root of $\underline{Q} $. 
		\end{enumerate}
		The slice Kleinian singularity has the same type as the quiver $\underline{Q}$. 
	\end{theorem}
	The following proposition is a consequence of \Cref{BS isotropic decomposition} and  \cite[Proposition 3.4]{wu2023namikawa}. 
	\begin{proposition}\label{prop: generator of NWi}
		Suppose $\tilde{v}\in \Sigma_0$. Let $\L_i$ be the codimension 2 leaf of $\M_0^0(v,w)$ corresponding to the representation type $\tau = (\beta_1,1;\cdots;\alpha_{i_t},m_t)$ as in \eqref{eq: an isot decomp}, and the coefficient of $\alpha_\infty$ in $\beta_1$ is 1.
		\begin{enumerate}
			\item The simple roots of $\hat\cartanh_i$ are in bijection with $\beta_2,\cdots,\alpha_{i_t}$. 
			\item The simple root in $\hat\cartanh_i$ corresponding to $\beta_i$ and $\beta_j$ lie in the same $\pi_1(\L_i)$-orbit if and only if $\beta_i = \beta_j$. 
			\item The simple roots in the folded root space $\cartanh_i$, as well as the generators of $W_i$, are in bijection with the \textbf{distinct} elements of $\beta_2,\cdots,\alpha_{i_t}$.
		\end{enumerate}
	\end{proposition}
	\begin{proposition}\label{prop: simple real cod 2 root}
		Suppose $\tilde{v}\in \Sigma_0$. Let $i\in Q_0^\infty$. The following are equivalent. 
		\begin{enumerate}
			\item there is a codimension 2 leaf of $\M_0^0(v,w)$ whose representation type contains $\alpha_i$ as a component;
			\item $(\alpha_i,\tilde{v}) = 0$. 
		\end{enumerate}
		Such a leaf is unique if it exists. 
	\end{proposition}
	\begin{proof}
		(1) $\implies$ (2) follows from (2) and (3) of \Cref{BS isotropic decomposition}. 
		Let us prove the other direction. 
		By the assumption $(\alpha_i,\tilde{v}) = 0$ it follows easily that 
		\begin{equation}\label{eq: dim of v-a is 2 less}
			2p(\alpha_i) + 2p(\tilde{v} -  \alpha_i) = 2p(\tilde{v} -  \alpha_i) = 2p(\tilde{v}) -2.
		\end{equation}
		If $\tilde{v}-  \alpha_i\in\Sigma_0$ then the representation type $(\tilde{v}- v_i\alpha_i,1;\alpha_i,v_i)$ corresponds to a codimension 2 leaf, by \Cref{CB strata dimension formula}. Otherwise, find a decomposition  
		\begin{equation}
			\tilde{v} - \alpha_i = c\alpha_i + \sum_{k=1}^l m_k\beta_k
		\end{equation}
		such that 
		\begin{enumerate}
			\item all $\beta_k \in \Sigma_0$, $\beta_k\neq \alpha_i$, and 
			\item $\sum_k p(\beta_k)$ is maximal among all such decompositions.
		\end{enumerate}
		We observe that if $(\beta_k,\beta_k) < 0$ then $p(n\beta_k) >n p(\beta_k)$ for $n\ge 2$, and if $(\beta_k,\beta_k) = 0$ then $np(\beta_k) > p(\beta_k)$ for $n\ge 2$. We conclude that $m_k\ge 2$ only if $\beta_k$ is a simple real root. 
		
		By construction $\sum_k p(\beta_k) \ge p(\tilde{v} - \alpha_i)$, and from the assumption $\tilde{v}\in \Sigma_0$ that $\sum_k p(\beta_k) + p(\alpha_i) <  p(\tilde{v} )$. 
		Note that $p(\beta_k)= 0$ when $\beta_k$ is a real root. 
		We conclude by \eqref{eq: dim of v-a is 2 less} that $(\beta_1,m_1;\cdots;\beta_l,m_l;\alpha_i,c+1)$ is a representation type that corresponds to a codimension 2 leaf. 
		
		Uniqueness of such a codimension 2 leaf follows from \cite[Corollary 6.6]{wu2023namikawa}. 
	\end{proof}
	\begin{proposition}\label{prop:a real root only appears in one leaf}
		Suppose $\tau_1 = (\beta_1,1; \cdots;\beta_s,1; \alpha_{i_1},m_1;\cdots;\alpha_{i_t},m_t)$ and 
		$\tau' = (\beta'_1,1 ;\cdots;\beta'_{s'},1; \alpha_{i_1'},m_1;\cdots;\alpha_{i_r'},m_r)$ are the representation types of two distinct codimension 2 leaves. 
		Then for any $\alpha_{i_j}$ and any $k$, we have 
		$(\alpha_{i_j}, \alpha_{i_k'}) = 0$ and $(\alpha_{i_j},\beta_k')=0$. 
	\end{proposition}
	\begin{proof}
		Without loss of generality let $j=1$. By  \Cref{prop: simple real cod 2 root}, $(\tilde{v},\alpha_{i_1})=0$, and all $\alpha_{i_k'} \neq \alpha_{i_1}$, so $(\alpha_{i_j},\alpha_{i_k'})\le 0$. Therefore, if $(\alpha_{i_1},\alpha_{i_k'})<0$ for some $k$, or $(\alpha_{i_1},\beta_k')\neq 0$ for some $k$, then 
		$(\alpha_{i_1},\beta_\ell') >0$ for some $\ell$. In particular, $i_1\in \Supp \beta_\ell'$. But by \cite[Theorem 5.6]{crawley2001geometry}, $(\beta_\ell'-\alpha_{i_1},\alpha_{i_1})\le -2$, contradicting that $(\alpha_{i_1},\beta_\ell') >0$. 
	\end{proof}
	
	\subsection{The canonical decomposition}
	Consider the quiver variety $\M_0^0(v,w) = \M_0^0(\tilde v)$. 
	We still assume $\tilde{v} \in \F_{Q^\infty}$. By \cite[Theorem 1.1]{su2006flatness}, the moment map $\mu$ is flat in this case, so $\M_0(v,w) = \M_0^0(v,w)$ by \cite[Proposition 2.5]{bezrukavnikov2021etingof}, and $\NamiWeyl(v,w)$ is just the Namikawa-Weyl group of $\M_0^0(v,w)$. We recall from \cite{cb2000decomposition} Crawley-Boevey's canonical decomposition of affine quiver varieties in the form we need. 
	\begin{theorem}[{\cite[Theorem 1.1, Proposition 1.2]{cb2000decomposition}}]\label{CB canonical decomposition}
		There exists a decomposition $\tilde v = \sum_{i=1}^k m_iv_i$, where $m_i$ are positive integers and $v_i \in \Sigma_0$, such that any decomposition of $v$ as a sum of elements in $\Sigma_0$ is a refinement of this decomposition. It is called the \emph{canonical decomposition} of $v$.
		We have 
		\begin{equation}
			\M_0^0( \tilde v) = \prod_{i=1}^k S^{m_i} \M^0_0(v_i),
		\end{equation}
		which we call the \emph{canonical decomposition} of $\M_0^0(\tilde v)$. Here $S^{m_i}$ denotes the symmetric power. 
		Moreover, the following are true. 
		\begin{enumerate}
			\item If $v_i$ is a real root then it is simple and $\M_0^0(v_i)$ is a point. 
			\item If $v_i$ is an isotropic imaginary root, then $v_i $ is the minimal imaginary root of some affine type quiver. We will write $v_i=\delta_i$ in this case, and $\M_0^0(m_i\delta_i) = S^{m_i}\M_0^0(\delta_i)$. 
			\item  If $v_i$ is a non-isotropic imaginary root, then $m_i=1$, and  $n v_i \in \Sigma_0$ for any $n\ge 1$. 
		\end{enumerate}
	\end{theorem}
	We are going to define the subgroup $W(v,w)_{\Irr}$ of $W(v,w)$ in \Cref{theorem:W vs NW}. First, we need to recall the algorithm of obtaining the canonical decomposition below, see \cite{cb2000decomposition}. 
	
	\begin{enumerate}
		\item If $\tilde{v} \in \Sigma_0$, we are done. 
		\item If not, then $Q_0^\infty$ is the disjoint union of two subsets, denoted $\mathcal{J},\mathcal{K}$, that satisfy the following properties. Let $v_\mathcal{J},v_\mathcal{K}$ be the restriction of $\tilde{v}$ to $\mathcal{J},\mathcal{K}$ respectively. Then
		$\mathcal{J},\mathcal{K}$ are connected by a single arrow, say $j\in \mathcal{J}$ is connected to $k\in \mathcal{K}$, 
		and one of the following holds:
		\begin{enumerate}
			\item $\tilde{v}_j = \tilde{v}_k = 1$; 
			\item $\mathcal{J}$ is an affine type quiver, $j$ is the extended vertex, $v_\mathcal{J} = m\delta$ where $\delta$ is the minimal imaginary root, $m\ge 2$, and $\tilde{v}_k = 1$. 
		\end{enumerate}
		\item For case (a) above, $\M_0^0(\tilde{v}) = \M_0^0(v_\mathcal{J}) \times \M_0^0(v_\mathcal{K})$. For (b), $\M_0^0(\tilde{v}) = S^m\M_0^0(\delta) \times \M_0^0(v_\mathcal{K})$. 
		\item The vector $v_\mathcal{K}$ either lies in $\F_{Q^\infty}$, or takes the following form 
		\begin{equation}    
			\begin{tikzpicture}[scale=0.9]
				\filldraw[black] (0,0) circle (1.5pt);
				\filldraw[black] (1,0) circle (1.5pt);
				\filldraw[black] (2,0) circle (1.5pt);
				\filldraw[black] (3,0) circle (1.5pt);
				\filldraw[black] (4,0) circle (1.5pt);
				\draw[-,thick] (.1,0)--(.9,0);
				\draw[-,thick] (1.1,0)--(1.9,0);
				\draw[-,thick] (3.1,0)--(3.9,0);

				\node at (2.5,0) {$\cdots$};	
				\node at (4.7,0)  {$v_\mathcal{K}'$};
				\node at (0,-0.4) {$\alpha_k$};	
				\node at (1,-0.4) {$\alpha_{k_1}$};
				\node at (2,-0.4) {$\alpha_{k_2}$};
				\node at (3,-0.4) {$\alpha_{k_r}$};
				\node at (4,-0.4) {$\alpha_{k_{r+1}}$};
				\node at (0,0.4) {$1$};	
				\node at (1,0.4) {$1$};
				\node at (2,0.4) {$1$};
				\node at (3,0.4) {$1$};
				\node at (4,0.4) {$1$};
				\draw (4.7,0) circle (1.2);
			\end{tikzpicture}
		\end{equation}
		where $(\alpha_{k_m} , v_\mathcal{K}) = 0$ for $1\le m \le r$ and $(\alpha_{k_{r+1}} , v_\mathcal{K}) \le -1$. Then the dimension vector 
		\begin{equation}
			v'_\mathcal{K} := v_\mathcal{K} - \alpha_k - \sum_{m=1}^r \alpha_{k_m} \in \F_{Q^\infty}
		\end{equation}  and $\M_0^0(v_\mathcal{K}) = \M_0^0(v'_\mathcal{K}) $. Similarly, we construct $v'_\mathcal{J}$ in case (a). 
		\item Now repeat the procedure starting from (1) with $v'_\mathcal{K}$ (and perhaps also $v'_\mathcal{J}$) instead of $\tilde{v}$. The procedure stops after finitely many iterations. 
	\end{enumerate}
	The following result follows directly from the construction above.
	\begin{proposition}\label{prop: decompose tildev into a string}
		Suppose $\tilde{v}\in\F_{Q^\infty}$. Then we can decompose $\tilde{v}$ as follows 
		\begin{equation}
			\begin{tikzpicture}[scale=1]
				\draw (0,0) circle (0.3);
				\node at (0,0) {$v_1$};
				\draw[-,thick] (.3,0)--(.6,0);
				\draw  (0.6,-0.2) rectangle ++(0.8,0.4); 
				\node at (1,0) {$\beta_1$};
				\draw[-,thick] (1.4,0)--(1.7,0);
				\draw (2,0) circle (0.3);
				\node at (2,0) {$v_2$};
				\draw[-,thick] (2.3,0)--(2.6,0);
				\draw  (2.6,-0.2) rectangle ++(0.8,0.4); 
				\node at (3,0) {$\beta_2$};
				\draw[-,thick] (3.4,0)--(3.7,0);
				\node at (4,0) {$\cdots$};
				\draw[-,thick] (4.3,0)--(4.6,0);
				\draw  (4.6,-0.2) rectangle ++(0.8,0.4); 
				\node at (5,0) {$\beta_{k-1}$};
				\draw[-,thick] (5.4,0)--(5.7,0);
				\draw (6,0) circle (0.3);
				\node at (6,0) {$v_k$};
			\end{tikzpicture}
		\end{equation}
		In the above picture, 
		\begin{enumerate}
			\item Each $v_i$ either lies in $\Sigma_0$ and is an imaginary root, or $v_i = m_i\delta_i$ such that $\delta_i$ is the minimal imaginary root of an affine type quiver.
			\item Each $\beta_i$ is either trivial (so that $v_i$ and $v_{i+1}$ are connected by a single arrow), or $\Supp \beta_i$ is a type $A$ quiver and $\beta_i$ is the longest root $(1,1,\cdots,1)$. 
		\end{enumerate}
		The canonical decomposition is $\M_0^0(v,w) = \prod_{i=1}^k \M_0^0(v_i) $, where $\M_0^0(v_i) = S^{m_i}\M_0^0(\delta_i)$ if  $v_i = m_i\delta$.  
	\end{proposition}
	
	\begin{definition}
		\begin{enumerate}
			\item Define the \textit{irrelevant subgroup} $W(v,w)_{\Irr}\subset W(v,w)$ to be the subgroup generated by 
			\begin{equation}\label{eq: generator of WvwIrr}
				\{s_j|j\in \Supp(\beta_r)\text{ for some $1 \le r \le k-1$, and } (\alpha_j,\tilde{v}) = 0\}. 
			\end{equation}
			\item Define the \textit{essential subgroup} $W(v,w)_{\Ess}\subset W(v,w)$ to be the subgroup generated by 
			\begin{equation}\label{eq:generators of WEss}
				Gen_{\Ess} = \{s_j|j\in \Supp(v_r)\text{ for some $1 \le r \le k$, and } (\alpha_j,\tilde{v}) = 0\}. 
			\end{equation}
		\end{enumerate}
	\end{definition}
	\begin{proof}[Proof of (1) of \Cref{theorem:W vs NW}]
		When $\tilde{v} \in \Sigma_0$ the canonical decomposition of $\M_0^0(v,w)$ contains no $\beta_r$ components, so $W(v,w)_{\Irr} = 1$ in this case. 
		In general, it is straightforward to check that the generators of $W(v,w)_{\Irr}$ and those of $W(v,w)_{\Ess}$ commute. By \Cref{prop:generators of Wvw}, there is a decomposition $W(v,w) = W(v,w)_{\Irr} \times W(v,w)_{\Ess}$.
	\end{proof}

	Now let us turn to $\NamiWeyl(v,w)$ and consider the codimension 2 leaves of $\M_0^0(v,w)$. If $X = \prod X_i$ is a finite product of symplectic singularities, then any codimension 2 leaf of $X$ takes the form $\L_j \times \prod_{i\neq j} X_i^{\reg}$, where $\L_j$ is a codimension 2 leaf of $X_j$. In view of the canonical decomposition of $\M_0^0(v,w)$, the following corollary is straightforward to check. 
	\begin{corollary}\label{cor: leaves rep types and NW under canonical decomp}
		Assume $\tilde{v} \in \F_{Q^\infty}$, and $\tilde v = \sum_{i=1}^k m_i v_i$ is the canonical decomposition of $\tilde v$, where $v_i$ is either a simple real root, or $v_i = m_i\delta_i$ for some affine minimal imaginary root $\delta$ and $m \ge 2$, or $v_i\in \Sigma_0$ is imaginary. (The notation is slightly different from that in \Cref{CB canonical decomposition}). Then 
		\begin{enumerate}
			\item $\M_0^0(v,w)$ has the following codimension 2 leaves. 
			\begin{enumerate}
				\item 
				If $\tau_{i,j}$ is a representation type of a codimension 2 leaf $\L_{i,j}\subset \M_0^0(v_i)$, where $v_i\in\Sigma_0$, then $\M_0^0(v,w)$ has a codimension 2 leaf with representation type 
				\begin{equation}
					\tilde \tau_{i,j} = (\cdots, v_{i-1},1; (\tau_{i,j}); v_{i+1},1; \cdots \underbrace{\delta_k,1;\cdots,\delta_k,1}_{\textrm{$m_k$ copies, if $v_k = m_k\delta_k$}}; \cdots).
				\end{equation}
				\item 
				If $v_i = m_i\delta_i$ and $m_i>1$, $\displaystyle\delta_i = \sum_{j\in\Supp \delta_i} n_j\alpha_j$, then $\M_0^0(v,w)$ has a codimension 2 leaf with representation type 
				\begin{equation}
					\tilde{\tau}_{i} = (\cdots, v_{i-1},1; (\alpha_j,n_j); \underbrace{\delta_i,1;\cdots;\delta_i,1}_{\text{$m_i -1$ copies}}
					; v_{i+1},1; \cdots).
				\end{equation}
				\item 
				If $v_i = m_i\delta_i$ and $m_i>1$, then  
				$\M_0^0(v,w)$ has a codimension 2 leaf with representation type 
				\begin{equation}
					\tilde{\tau}_{i,0} = (\cdots, v_{i-1},1; \underbrace{\delta_i,1;\cdots;\delta_i,1}_{\text{$m_i -2$ copies}};\delta_i,2
					; v_{i+1},1; \cdots).
				\end{equation}
			\end{enumerate}
			\item The group $\NamiWeyl(v,w)$ has the following decomposition: 
			\begin{equation}
				\NamiWeyl(v,w) = \prod_{i=1}^k \NamiWeyl(v_i)
			\end{equation}
			where $\NamiWeyl(v_i)$ denotes the Namikawa-Weyl group of $\M_0^0(v_i)$. If $v_i = m_i\delta_i,m_i\ge 2$, then 
			\begin{equation}
				\NamiWeyl(v_i) = W_i \times \ZZ/2\ZZ
			\end{equation} 
			where $W_i$ is the finite Weyl group of the same type as $\delta_i$. 
			\item The decomposition \Cref{Ivan's decomposition of cartan space} takes the form $H^2(\M_0^\theta(v,w),\C) =$ 
			\begin{equation}\label{eq: ivans decomposition under canonical decomposition}
				H^2(\M_0^0(v,w)^{\reg},\C) \oplus \bigoplus_{\substack{v_i\in\Sigma_0\\ 1\le j \le  n_j}}\cartanh_{i,j} \oplus 
				\bigoplus_{\substack{v_i =m_i\delta_i \\m_i>1}} \cartanh_i 
				\oplus
				\bigoplus_{\substack{v_i =m_i\delta_i \\m_i>1}}\C
			\end{equation}
			where $\L_{i,1},\cdots,\L_{i,n_i}$ are the codimension 2 leaves of $\M_0^0(v_i)$, 
			$\cartanh_{i,j}$ is associated to $\L_{i,j}$, and 
			$\cartanh_i$ is isomorphic to the finite type Cartan space of $\delta_i$. 
		\end{enumerate}
	\end{corollary}
	
	Next, we produce a map from the set $Gen_{\Ess}$ to $\prod_{i=1}^k \NamiWeyl(v_i) $. 
	For each $v_i$, it follows from the construction of the canonical decomposition that if $j \in \Supp(v_i)$ is connected to $\beta_{i-1}$ or $\beta_i$, then $(v_i)_{j} = 1$. Therefore, we pick one such vertex in each $\Supp v_i$ and view it as a framing in $\M_0^0(v_i)$. Denote this vertex by $i_\infty$. Exceptionally, if $\infty\in \Supp v_i$ then we pick $i_\infty = \infty$, even if it is not connected to $\beta_i$ or $\beta_{i-1}$. 
	\begin{lemma}
		\begin{enumerate}
			\item For each $1\le i \le k$,  $(\alpha_{i_\infty},\tilde{v}) < 0$ unless $i_\infty = \infty$; in particular, $s_{i_\infty}\not\in Gen_{\Ess}$. 
			\item If $j\in \Supp(v_i)$ and $s_j\in Gen_{\Ess}$, then there is a unique codimension 2 leaf of $\M_0^0(v_i)$ whose representation type contains $\alpha_j$ as a component. 
		\end{enumerate}
		
	\end{lemma}
	
	\begin{proof}
		If $i_\infty = \infty$ we are done since $s_\infty \not \in Gen_{\Ess}$ anyways. Otherwise, note that $v_i$ either lies in $\Sigma_0$ or $v_i = m\delta$, so $(\alpha_{i_\infty},v_i ) \le 0$. Since $i_\infty$ is connected to a vertex outside $\Supp v_i$, we have $(\alpha_{i_\infty},\tilde{v} ) < 0$. This proves (1). Now (2) follows from (1) and \Cref{prop: simple real cod 2 root}. 
	\end{proof}
	
	\begin{definition}\label{def: set theoretic iota}
		Let $j\in \Supp(v_i)$ and $s_j\in Gen_{\Ess}$. Define $\iota$ to be the map $Gen_{\Ess} \to \NamiWeyl(v_i)$ that maps $s_j$ to the generator of $\NamiWeyl(v_i)$ corresponding to $\alpha_j$, see \Cref{prop: generator of NWi}.  
	\end{definition}
	We will show that $\iota$ extends to a group monomorphism in \Cref{section:W action on H2}.
	\subsection{Examples}
	We give some examples that shed some light on the relation between $W(v,w)$ and $\NamiWeyl(v,w)$. 
	\begin{example}[{\cite[Theorem 5.4]{mcgerty2019springer}}]{\label{MN weyl group dynkin quiver}}
		Assume $Q$ is a simply-laced Dynkin quiver and $\tilde{v}\in\Sigma_0$. Then 
		$\NamiWeyl(v,w) = W(v,w)$. 
	\end{example}
	
	\begin{example}[{\cite[Proposition 2.2]{bellamy2020birational}}]
		Suppose $Q$ is a simply-laced  affine type quiver, and let $\delta$ be the minimal positive imaginary root. 
		Suppose $v = n\delta$ and $w$ is 1 over the extended vertex, 0 elsewhere. Then $\NamiWeyl(v,w) = W_{Q_{fin}} \times \ZZ/2\ZZ$ while $W(v,w) = W(v,w)_{\Ess} =  W_{Q_{fin}}$. 
		In this example, $\tilde{v}\in \F_{Q^\infty}$ but not $\Sigma_0$, and there is a natural embedding $W(v,w)\hookrightarrow \NamiWeyl(v,w)$. 
	\end{example}
	
	\begin{example}
		The following example is a variant of \cite[5(v)]{braverman2019ring}. 
		Consider the quiver below. 
		
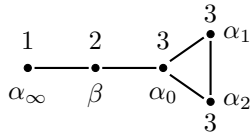
\begin{center}
			\begin{tikzpicture}[scale=0.9]
				\filldraw[black] (0,0) circle (1.5pt);
				\filldraw[black] (1,0) circle (1.5pt);
				\filldraw[black] (2,0) circle (1.5pt);
				\filldraw[black] (2.7,0.5) circle (1.5pt);
				\filldraw[black] (2.7,-0.5) circle (1.5pt);
				\node at (2,0) (alpha3) {$\ $};
				\draw[-,thick] (0.1,0)--(0.9,0);
				\draw[-,thick] (1.1,0)--(1.9,0);
				\draw[-,thick] (2.1,0.1)--(2.6,0.45);
				\draw[-,thick] (2.1,-0.1)--(2.6,-0.45);
				\draw[-,thick] (2.7,0.4)--(2.7,-0.4);
				\node at (0,-0.4) {$\alpha_\infty$};	
				\node at (3.1,0.5) {$\alpha_1$};	
				\node at (3.1,-0.5) {$\alpha_2$};	
				\node at (2.7,0.8) {$3$};	
				\node at (2.7,-0.8) {$3$};	
				\node at (1,-0.4) {$\beta$};
				\node at (2,-0.4) {$\alpha_0$};
				\node at (0,0.4) {$1$};	
				\node at (1,0.4) {$2$};
				\node at (2,0.4) {$3$};
			\end{tikzpicture}
		\end{center}
		where we view $\alpha_\infty$ as a framing, so that $\tilde{v} = \alpha_\infty+2\beta +3(\alpha_0+\alpha_1+\alpha_2)$. It is not hard to see $\Tilde{v}\in\Sigma_0$, and $\M_0^0(v,w)$ has 2 codimension 2 leaves. The corresponding representation types are 
		\begin{equation}
			\begin{split}
				&\tau_1 = (\alpha_\infty, 1;\beta,2;
				\delta,1;\delta,1;\delta,1); \\
				&\tau_2 = (\alpha_\infty+2\beta+3\alpha_0+2\alpha_1+2\alpha_2,1; \alpha_1,1;\alpha_2,1).
			\end{split}
		\end{equation}
		The group $\NamiWeyl(v,w)$ is the product of a type $G_2$ Weyl group with a type $A_2$ Weyl group, while $W(v,w) = W(v,w)_{\Ess}  = \ZZ/2\ZZ \times S_3$, generated by the reflections at $\beta$ and at $\alpha_1,\alpha_2$. The $\ZZ/2\ZZ$ component properly embeds into the $G_2$ component of $\NamiWeyl(v,w)$, and the $S_3$ maps isomorphically onto the type $A_2$ Weyl group component of $\NamiWeyl(v,w)$. 
	\end{example}

	\begin{example}
		Consider the following quiver, where we again regard $\alpha_\infty$ as the framing.    
		\begin{equation}    
			\begin{tikzpicture}[scale=0.9]
				\filldraw[black] (0,0) circle (1.5pt);
				\filldraw[black] (1,0) circle (1.5pt);
				\filldraw[black] (2,0) circle (1.5pt);
				\filldraw[black] (3,0) circle (1.5pt);
				\filldraw[black] (4,0) circle (1.5pt);
				\draw[-,thick] (.1,0)--(.9,0);
				\draw[-,thick] (.1,0.1)--(.9,0.1);
				\draw[-,thick] (.1,-0.1)--(.9,-0.1);
				\draw[-,thick] (1.1,0)--(1.9,0);
				\draw[-,thick] (2.1,0)--(2.9,0);
				\draw[-,thick] (3.1,0)--(3.9,0);
				\draw[-,thick] (3.1,0.1)--(3.9,0.1);
				\draw[-,thick] (3.1,-0.1)--(3.9,-0.1);
				\node at (0,-0.4) {$\alpha_\infty$};	
				\node at (1,-0.4) {$\alpha_1$};
				\node at (2,-0.4) {$\alpha_2$};
				\node at (3,-0.4) {$\alpha_3$};
				\node at (4,-0.4) {$\alpha_4$};
				\node at (0,0.4) {$1$};	
				\node at (1,0.4) {$1$};
				\node at (2,0.4) {$1$};
				\node at (3,0.4) {$1$};
				\node at (4,0.4) {$1$};
			\end{tikzpicture}
		\end{equation}
		The reflection $s_2$ at $\alpha_2$ lies in $W(v,w)$. The canonical decomposition of $\M_0^0(v,w)$ is 
		\begin{equation}
			\M_0^0(v,w) = \M_0^0(\alpha_1 + \alpha_\infty) \times \M_0^0(\alpha_2) \times \M_0^0(\alpha_3+\alpha_4)
		\end{equation}
		where $\M_0^0(\alpha_2) $ is just a point, and $\M_0^0(\alpha_1 + \alpha_\infty), \M_0^0(\alpha_3 + \alpha_4)$ both have no codimension 2 leaves. 
		Here $W(v,w)  = W(v,w)_{\Irr} = \ZZ/2\ZZ$, and $\NamiWeyl(v,w) = W(v,w)_{\Ess} =  \{1\}$.  
	\end{example}
	\section{The symplectic Springer action}\label{subsection: The symplectic Springer action}
	We follow \cite[Definition 4.7]{mcgerty2019springer} and recall the definition of the symplectic Springer action of $\NamiWeyl(v,w)$ on $H^*(\M_0^\theta(v,w),\C)$. 
	
	Let $Y\to X$ be a conical resolution of symplectic singularities. 
	Write $m_Y: \mathcal{Y}\to B_Y$ for the universal graded Poisson deformation of $Y$, where $B_Y = H^2(Y,\C)$. 
	Recall that $B_X = B_Y/\NamiWeyl$ is the base of the universal graded Poisson deformation of $X$. Denote the universal graded Poisson deformation of $X$ by $\mathcal{X}$, and set $\mathcal{X}' = \mathcal{X} \times_{B_X} B_Y$. We have the following commutative diagram, where the right square is Cartesian: 
	\begin{equation}\label{diag: Xprime to By}
		\begin{tikzcd}
			\mathcal{Y} \arrow["\Pi",r] \arrow[d,"m_Y"] & \mathcal{X}' \arrow[r] \arrow[d,"m'_X"] & \mathcal{X} \arrow[d,"m_X"]\\ 
			B_Y \arrow["\id",r]&  B_Y\arrow[r,"q"] & B_X
		\end{tikzcd}
	\end{equation}
	
	In \cite{namikawa2015poissonbirational}, Namikawa defined a closed subset $D\subset B_Y$; it is the locus over which $\Pi$ is not an isomorphism; equivalently, it consists of $b\in B_Y$ such that $m_Y\inverse(b)$ is not affine. 
	By \cite[Main Theorem]{namikawa2015poissonbirational},  $D$ is a union of hyperplanes and is invariant under the action of $\NamiWeyl(v,w)$ on $B_Y$, and $\NamiWeyl$ acts freely on $ B_Y\setminus D$. 
	
	Set $B_Y^\circ := B_Y\setminus D$ and $B_X^\circ = B_Y^\circ/\NamiWeyl$. Set $\mathcal{Y}^\circ := m_Y\inverse(B_Y^\circ)$, $\mathcal{X}^\circ = m_X\inverse(B_X^\circ)$, and $(\mathcal{X}')^\circ = (m_X')\inverse(B_Y^\circ)$. Let $\Pi^\circ$ denote the restriction of $\Pi$ to $\mathcal{Y}^\circ$. 
	It is an isomorphism by the definition of $D$. 
	Finally let $m_Y^\circ,(m_X')^\circ, m_X^\circ$ and $q^\circ$ be the restrictions of $m_Y,m_X',m_X$ and $q$ to $\mathcal{Y}^\circ, (\mathcal{X}')^\circ , \mathcal{X}^\circ$ and $B_Y^\circ$ respectively. 
	Restricting diagram \eqref{diag: Xprime to By} over $B_X^\circ$ and $B_Y^\circ$, we get 
	\begin{equation}
		\begin{tikzcd}
			\mathcal{Y}^\circ \arrow["\Pi^\circ",r] \arrow[d, "m_Y^\circ"] & (\mathcal{X}')^\circ \arrow[d, "(m_X')^\circ"] \arrow[r] & \mathcal{X}^\circ \arrow[d, "m_X^\circ"] \\ 
			B_Y ^\circ\arrow["\id",r]&  B_Y ^\circ \arrow[r,"q^\circ"]& B_X^\circ
		\end{tikzcd}
	\end{equation} 
	It follows that there is a free $\NamiWeyl$-action on $(\mathcal{X}')^\circ$, and hence on $\mathcal{Y}^\circ$. 
	The map $m_Y^\circ$ is $\NamiWeyl$-equivariant. So for any $b^\circ\in B_Y^\circ$ and any $\gamma \in \NamiWeyl$, there is an isomorphism
	\begin{equation}
		\Psi_\gamma: m_Y\inverse(b^\circ) \xrightarrow{\sim} m_Y\inverse(\gamma b^\circ)
	\end{equation} 
	By \cite[Proof of Theorem 1.1, Step 2, part (i)]{namikawa2010poisson}, the map $m_Y: \mathcal{Y} \to B_Y$ is a $C^\infty$ trivial fiber bundle. 
	Therefore, for any $b,b' \in B_Y$, there is a natural isomorphism $\mathrm{top}: H^*(m_Y\inverse(b)) \cong H^*(m_Y\inverse(b'))$, where $\mathrm{top}$ stands for topological. 
	Note $Y = m_Y\inverse(0)$. 
	\begin{definition}[{\cite[Definition 4.7]{mcgerty2019springer}}]\label{def: symp springer action}
		The \textit{symplectic Springer $\NamiWeyl$-action} on $H^*(Y,\C)$ is given as follows: an element $\gamma \in \NamiWeyl$ acts by the composition 
		\begin{equation}
			\gamma*: H^*(Y,\C) \xrightarrow{\mathrm{top}} H^*(m_Y\inverse(\gamma b^\circ),\C) \xrightarrow{\Psi_\gamma^*} H^*(m_Y\inverse( b^\circ),\C) \xrightarrow{\mathrm{top}}H^*(Y,\C) .
		\end{equation}
	\end{definition}
	
	Specifically, let $Y = \M_0^\theta(v,w)$ and $X = \M_0^0(v,w)$, and assume $Y\to X$ is a conical symplectic resolution of singularities. 
	Recall the map $\kappa: \mathfrak{p} \to B_Y$ defined in \eqref{diag: kappa cartesian}. 
	Set $\mathfrak{p}^\circ = \kappa\inverse(B_Y^\circ)$. The following Cartesian diagram is obtained by restricting \eqref{diag: kappa cartesian} to $B_Y^\circ$: 
	\begin{equation}
		\begin{tikzcd}
			\M_{\mathfrak{p^\circ}}^\theta(v,w) \arrow[r] \arrow[d,"\mu"] & \mathcal{Y^\circ} \arrow[d,"m_Y"] \\ 
			\mathfrak{p}^\circ \arrow[r,"\kappa"] & B_Y^\circ
		\end{tikzcd}
	\end{equation}
	\begin{remarklabeled}\label{remark springer equals natural on H2}
		Recall from \Cref{remark natural action NW on namikawacartan} that $\NamiWeyl(v,w)$ acts naturally on $H^2(\M_0^\theta(v,w),\C)$. 
		For any $\gamma\in \NamiWeyl(v,w)$, it is shown in \cite[Lemma 4.8]{mcgerty2019springer} that the symplectic Springer action $\gamma*$ coincides with the natural action of $\gamma$ on $H^2(\M_0^\theta(v,w),\C)$. 
	\end{remarklabeled}
	\section{Maffei's action and tautological line bundles}\label{section:W action on H2}
	In this section we recall the construction of the Maffei's isomorphisms, and examine the Maffei's action of $W(v,w)$ on $H^*(\M_0^\theta(v,w),\C)$. We will compare the Maffei's action with the natural $\NamiWeyl(v,w)$ action on $H^2(\M_0^\theta(v,w),\C)$. 
	\begin{thmdef}\label{thm: mff isomorphism and action}
		Suppose $(\theta,\lambda)$ is generic, and let $\sigma\in W_Q$. 
		\begin{enumerate}
			\item (\cite[Section 3]{maffei2002remark}) There is an isomorphism of varieties 
			\begin{equation}
				\Phi_\sigma:\M_\lambda^{ \theta} (v,w) \xrightarrow{\sim} \M_{\sigma^*\lambda}^{\sigma^* \theta} (\sigma\bullet v,w).
			\end{equation}
			Here $\sigma\bullet v$ is defined by $\sigma\bullet v + \alpha_\infty = \sigma \tilde{v}$. 
			Moreover, if $\sigma_1,\sigma_2\in W_Q$ then $\Phi_{\sigma_1\sigma_2} = \Phi_{\sigma_1} \circ \Phi_{\sigma_2}$.
			
			\item (\cite[Corollary 49]{maffei2002remark} and \cite[Corollary 9.3]{nakajima1994instantons}) There is a $W(v,w)$-action on $H^n(\M_\lambda^\theta (v,w),\ZZ)$, and hence on $H^n(\M_\lambda^\theta (v,w),\C)$. 
		\end{enumerate}
		We call $\Phi_\sigma$ the \emph{Maffei's isomorphism} and the above action on $H^n(\M_\lambda^\theta (v,w),\ZZ)$ the \emph{Maffei's action}. 
	\end{thmdef}
	We will focus on the case where $\theta$ is generic, $\lambda = 0$, and $n=2$. 
	\subsection{Maffei's isomorphism}	\label{subsec: construction of Maffei iso}
	The construction of Maffei's isomorphism will be useful later, so let us recall it in the case of a simple reflection, following \cite[Section 3.1]{maffei2002remark}. Fix $Q,v,w$, pick $j\in Q_0$, and let $\theta_j >0$. 
	Write $\theta'  = s_j^*{\theta} $ and $\tilde v 
	' = s_j \tilde v$. We are going to define the isomorphism 
	\begin{equation}
		\Phi_{s_j}: \M_\mathfrak{p}^{\theta} (\tilde{v}) \xrightarrow{\sim} \M_\mathfrak{p}^{\theta'} (\tilde v')
	\end{equation}
	that restricts to an isomorphism 
	\begin{equation}\label{eq: maffei for one level}
		\M_\lambda^{\theta} (\tilde{v}) \xrightarrow{\sim} \M_{\lambda'}^{\theta'} (\tilde v')
	\end{equation}
	for any $\lambda\in\mathfrak{p}$, where $\lambda' = \sigma^*\lambda$. 
	
	Without loss of generality assume $j$ is a source. 
	For $k \in Q_0^\infty$, let $V_k,V_k'$ be the vector spaces with dimensions $\tilde{v}_k,\tilde{v}'_k$. Note that when $k\neq j$ we may identify $V_k '= V_k$. Define 
	\begin{equation}\label{eq:def of Tj}
		T_j = \bigoplus_{a, t(a) = j} V_{h(a)}.
	\end{equation}
	By the definition of the $W_Q$-action, $\dim T_j = \tilde v_j + \tilde v'_j$. 
	For an element $r'\in T^*R(Q^\infty,\tilde{v}')$ we write 
	\begin{equation}
		r' = (x_a',y_a')_{a\in Q_1^\infty}
	\end{equation}
	similarly to \eqref{eq:expression of a representation}. 
	Set $\mu': T^*R(Q^\infty,\tilde{v}') \to \prod_{i\in Q_0} \gl(V_i')$ to be the standard moment map.  
	Define 
	\begin{equation}\label{eq: def of Aj and Bj}
		A_j = \bigoplus_{a,t(a)=j} x_a: V_j \to T_j, \ \ 
		B_j = \bigoplus_{a,t(a)=j} y_a: T_j \to V_j.
	\end{equation}
	Similarly define $A_j': V_j'\to T_j$ and $B_j':T_j\to V_j'$.

	Let $Z$ be the subset of $\mu\inverse(\mathfrak{p}) \times (\mu')\inverse(\mathfrak{p})$ consisting of all elements $(r,r')$ such that 
	\begin{enumerate}
		\item $x_a=x'_a,y_a=y'_a$ if $t(a)\neq j$.
		\item The linear maps $A_j',B_j$ fit into a short exact sequence 
		\begin{equation}
			0\to V_j' \xrightarrow{A_j'} T_j \xrightarrow{B_j} V_j \to0 .
		\end{equation}
		\item  $A_j'B_j' = A_jB_j - \lambda_j \id_{T_j}$ if $ \mu(r) =\lambda$. 
		\item $\mu'(r') = \lambda' = \sigma^*\lambda$ if $ \mu(r) =\lambda$. 
	\end{enumerate}
	Denote by $p$ (resp. $p'$) the projection $Z\to \mu\inverse(\mathfrak{p}) $ (resp $Z\to (\mu')\inverse(\mathfrak{p}) $ ), $(r,r')\mapsto r$ (resp $r'$). 
	Define $G_{j,v} = \GL(V_j)\times \GL(V_j')\times \prod_{i\neq j}\GL(V_i)$; there are natural projections $G_{j,v}\to \GL(v)$ and $G_{j,v}\to \GL(v')$, so that $\mu\inverse(\mathfrak{p})$ and $(\mu')\inverse(\mathfrak{p})$ carry $G_{j,v}$ actions. The projection $p$ and $p'$ are $G_{j,v}$-equivariant. 
	
	By \cite[Lemma 30]{maffei2002remark}, $p\inverse(\mu\inverse(\mathfrak{p})^{\theta-ss}) = (p')\inverse ((\mu')\inverse(\mathfrak{p})^{s_j^*\theta-ss})$. 
	Denote $Z^{ss}: = p\inverse(\mu\inverse(\mathfrak{p})^{\theta-ss})$. Moreover, by \cite[Lemma 33]{maffei2002remark}, the restrictions 
	\begin{equation}\label{eq: the projections p and p' from Zss}
		\begin{split}
			&p: Z^{ss} \to \mu\inverse(\mathfrak{p})^{\theta-ss}  \\
			&p': Z^{ss} \to (\mu')\inverse(\mathfrak{p})^{s_j^*\theta-ss}
		\end{split}
	\end{equation}
	are principal $\GL(V_j')$- and $\GL(V_j)$-bundles respectively. It follows that we have isomorphisms 
	\begin{equation}
		\M_\mathfrak{p}^\theta (v,w) \cong Z^{ss}/G_{j,v} \cong \M_\mathfrak{p}^{s_j^*\theta} (s_j \bullet v,w)
	\end{equation}
	and $\Phi_{s_j}$ is defined to be the composition of the above isomorphisms. From the construction, $\Phi_{s_j}$ restricts to an isomorphism \eqref{eq: maffei for one level} for any $\lambda\in \mathfrak{p}$. We will denote the latter isomorphism by $\Phi_{s_j}$ as well. 
	
	\subsection{Tautological line bundles}
	Let $\theta$ be generic so that $G$ acts freely on  $(T^*R)^{\theta-ss}$. Let $\chi\in\ZZ^{Q_0}$ be a character of $G$. We view $\chi$ as an element of $\cartanh_Q^*$. 
	If $X,Y$ are $G$-varieties and $G$ acts freely on $X$, then we write $X\times^G Y = (X\times Y) /G$ by the diagonal action; it is a $Y$-bundle over $X/G$. 
	\begin{definition}
		The line bundle over $\M^\theta_\mathfrak{p}(v,w) $, 
		whose total space is the homogeneous bundle
		\begin{equation}
			\mu\inverse(\mathfrak{p})^{\theta -ss} \times^G \C_\chi
		\end{equation}
		where $G$ acts on the one dimensional vector space $\C_\chi$ by the character $\chi$, is called a \textit{tautological line bundle}. We denote it by $\mathcal{O}(\chi,\M^\theta_\mathfrak{p}(v,w))$. 
		
		For any $\lambda\in\mathfrak{p}$, we can restrict $\mathcal{O}(\chi,\M^\theta_\mathfrak{p}(v,w))$ to $\M_\lambda^\theta(v,w)$. The resulting line bundle will be denoted $\mathcal{O}(\chi,\M^\theta_\lambda (v,w))$; its total space is 
		\begin{equation}
			\mu\inverse(\lambda)^{\theta -ss} \times^G \C_\chi.
		\end{equation}
	\end{definition}
	\begin{theorem}\label{theorem: kappa is chern class and surjective}
		The following are true. 
		\begin{enumerate}
			\item (\cite[Proposition 3.2.1.]{losev2012isomorphisms}) The map $\kappa$ defined in \eqref{diag: kappa cartesian} coincides with the map 
			\begin{equation}
				\begin{split}
					\ZZ^{Q_0} \otimes\C \mapsto H^2(\M_0^\theta(v,w),\C)\\
					\chi \otimes \lambda \mapsto \lambda c_1(\mathcal{O}(\chi,\M^\theta_0(v,w)))
				\end{split}
			\end{equation}
			where $c_1$ denotes the first Chern class of a line bundle. 
			\item (\cite[Theorem 1.2]{mcgerty2018kirwan}) The map $\kappa$ is surjective.
		\end{enumerate}
	\end{theorem}
	Thanks to the surjectivity of $\kappa$, to study the $W(v,w)$ action on $H^2(\M_0^\theta(v,w),\C)$, it suffices to study the corresponding action on the tautological line bundles. 	
	\begin{proposition}\label{prop: tautological line bundle and maffei}
		Let $\chi\in \ZZ^{Q_0}$, $j\in Q_0$. Then 
		\begin{equation}
			\Phi_{s_j}^* \sheafO(\chi, \M_0^\theta(v,w)) = \sheafO(s_j^*\chi, \M_0^{s_j^*\theta}(s_j\bullet v,w)).
		\end{equation}
	\end{proposition}
	\begin{proof}
		Recall $Z^{ss}$ defined in the construction of Maffei's isomorphism; let $Z_0^{ss}\subset Z^{ss}$ be the subset consisting of pairs $(r,r')$ such that $\mu(r) = 0,\mu'(r')=0$. 
		We are going to construct a line bundle on $Z_0^{ss}$ which is isomorphic to the pullback of $\sheafO(\chi, \M_0^\theta(v,w))$ under $p$ and also the pullback of $\sheafO(s_j^*\chi, \M_0^{s_j^*\theta}(s_j\bullet v,w))$ under $p'$, where $p,p'$ denote the restrictions of the projections in \eqref{eq: the projections p and p' from Zss} to $Z_0^{ss}$. This will imply the proposition by the construction of $\Phi_{s_j}$. 
		
		First we define a nonzero function on $Z_0^{ss}$. Fix an injection $A_{j,0}':V_j'\to T_j$ and a surjection $B_{j,0}\to V_j$ such that $0\to V_j' \xrightarrow{A_{j,0}'} T_j \xrightarrow{B_{j,0}} V_j \to0 $ is a short exact sequence. 
		For any element in $Z_0^{ss}$, there is an element $g\in \GL(T_j)$ such that $ \im(A_j') = g \im (A_{j,0}')$. We then define $\alpha_j'\in\GL(V_j')$ and $\beta_j\in\GL(V_j)$ by the following commutative diagram 
		\begin{equation}
			\begin{tikzcd}
				0 \arrow[r] & V_j' \arrow[r, "{A_{j,0}'}"] \arrow[d, "\alpha_j'"] & T_j \arrow[r, "{B_{j,0}}"] \arrow[d, "g"] & V_j \arrow[r] \arrow[d, "\beta_j"] & 0 \\
				0 \arrow[r] & V_j' \arrow[r, "A_j'"]                              & T_j \arrow[r, "B_j"]                      & V_j \arrow[r]                      & 0
			\end{tikzcd}.
		\end{equation}
		For $(r,r')\in Z_0^{ss}$, define 
		\begin{equation}
			F(r,r') = \det(\alpha_j')\det(\beta_j)\det(g)\inverse.
		\end{equation}
		Note that $F(r,r')$ is independent of the choice of $A_{j,0}',B_{j,0}$ and $g$.
		
		Now let $\chi = (\chi_i)_{i\in Q_0}$ be a character of $\GL(v)$ and $s_j^*\chi = (\chi_i - c_{ij}\chi_j)_{i\in Q_0}$ be the corresponding character of $\GL(v')$. Define a character $\tilde\chi = \chi\times 1$ of $G_{j,v}$, i.e. 
		\begin{equation}
			G_{j,v} = \GL(V_j)\times \GL(V_j')\times \prod_{i\neq j}\GL(V_i)\ni(g_j,g_j',g_i)_{i\neq j} \mapsto \chi(g_j,g_i)_{i\neq j}.
		\end{equation}
		Consider the $G_{j,v}$-space $Z_0^{ss}\times \C_{\tilde\chi}$; it is easy to see that the map 
		\begin{equation}
			\begin{split}
				\tilde{p}: Z_0^{ss}\times \C_{\tilde\chi} &\to \mu\inverse(0)^{\theta -ss} \times \C_\chi\\
				(r,r',t)&\mapsto (r,t)
			\end{split}
		\end{equation}
		is $G_{v}$-equivariant, and induces an isomorphism of line bundles $Z_0^{ss}\times^{G_j,v} \C_{\tilde\chi} = \sheafO(\chi, \M_0^\theta(v,w))$. 
		On the other hand, we observe that for $\tilde g = (g_j,g_j',g_i)_{i\neq j}\in G_{j,v}$, 
		\begin{equation}
			F(\tilde g . (r,r')) = F(r,r') \det(g_j)\det(g_j')\inverse \prod_{i\neq j}\det(g_i)^{-c_{ij}}.
		\end{equation}
		It is then straightforward to verify that the map
		\begin{equation}
			\begin{split}
				\tilde{p}: Z_0^{ss}\times \C_{\tilde\chi} &\to \mu\inverse(0)^{s_j^*\theta -ss} \times \C_{s_j^*\chi}\\
				(r,r',t)&\mapsto (r',F(r,r')^{-\chi_j} t )
			\end{split}
		\end{equation}
		is $G_{v'}$-equivariant, and induces an isomorphism of line bundles $Z_0^{ss}\times^{G_j,v} \C_{\tilde\chi} = \sheafO(s_j^*\chi, \M_0^{s_j^*\theta}(s_j\bullet v,w))$. 
		This finishes the proof of the proposition. 
	\end{proof}
\subsection{Maffei's action on the cohomology}
	Let $\star$ denote the Maffei's action of $W(v,w)$ on $H^2(\M_0^\theta(v,w),\C)$. Let us recall its construction, following \cite[Section 5]{maffei2002remark}. We need the hyper-K\"{a}hler construction of quiver varieties, see \cite[Section 2]{nakajima1994instantons} for details. 
	Namely, the vector space $T^*R$ has a natural hyper-K\"{a}hler structure. Let $U = \prod_{i\in Q_0} U(v_i)$ be the maximal compact subgroup of $G$, $\mf{u}$ be its Lie algebra; 
	let $\mf{p}_{hk} := (\mf{u}^*)^U \otimes_\RR \RR^3 \cong \RR^3 \otimes_\ZZ \ZZ^{Q_0}$. 
	The group $U$ acts on $T^*R$ and respects the hyper-K\"{a}hler structure. 
	Let $\mf{p}_{hk}^{\reg}$ be the open subset of $\mf{p}_{hk}$ containing all $\boldsymbol{\zeta}$ such that $\boldsymbol{\zeta}\cdot u \neq 0$ for all $0< u \le v$. 
	Let $\boldsymbol{\mu}: T^*R \to \mf{p}_{hk}$ denote the hyperkahler moment map. Then by \cite[Lemma 48]{maffei2002remark}, the natural morphism 
	\begin{equation}
		\underline{\boldsymbol{\mu}}: \boldsymbol{\mu}\inverse (\mf{p}_{hk}^{\reg}) / U \to \mf{p}_{hk}^{\reg}
	\end{equation}
	is a locally trivial fibration. 
	Recall $\underline{\boldsymbol{\mu}}_*$ denotes the pushforward functor for the category of constructible sheaves with $\ZZ$-coefficients. 
	The base $\mf{p}_{hk}^{\reg}$ is simply-connected. Therefore, the local system 
	\begin{equation}
		\underline{\boldsymbol{\mu}}_* \underline{\ZZ}_{\boldsymbol{\mu}\inverse (\mf{p}_{hk}^{\reg}) / U}
	\end{equation}
	on $\mf{p}_{hk}^{\reg}$ has trivial monodromy, see e.g. \cite[Theorem 1.7.9]{achar2021perverse}. 
	In particular, for any $\boldsymbol{\zeta},\boldsymbol{\zeta}'\in \mf{p}_{hk}^{\reg}$, the cohomology groups 
	$H^n(\boldsymbol{\mu}\inverse (\boldsymbol{\zeta}) / U,\ZZ)$ and $H^n(\boldsymbol{\mu}\inverse (\boldsymbol{\zeta}') / U,\ZZ)$ are naturally identified. 
	By \cite[Proposition 2.4.1]{nakajima2001quiver}, for generic $\theta\in\ZZ^{Q_0}, \lambda\in \C^{Q_0}$, there is an inclusion $\boldsymbol{\mu}\inverse((i\theta,\lambda)) \hookrightarrow \mu\inverse(\lambda)^{\theta-ss}$, inducing a homeomorphism $\boldsymbol{\mu}\inverse ((i\theta,\lambda)) / U \xrightarrow{\sim} \M_\lambda^\theta(v,w)$. 
	Therefore, for any generic $(\theta,\lambda)$ and $(\theta',\lambda')$, and for any $n\ge 0$, there is a natural isomorphism 
	\begin{equation}
		\mathrm{top}: H^n(\M_\lambda^\theta(v,w),\ZZ)\xrightarrow{\sim }H^n(\M_{\lambda'}^{\theta'}(v,w),\ZZ)
	\end{equation} 
	where top stands for ``topological". 
	\begin{definition}\label{def: maffei action}
		Let $\theta$ be generic and $\lambda \in \mathfrak{p}$ be generic. 
		For any $\sigma \in W(v,w)$, the Maffei's action $\sigma\star$ is defined by the composition 
		\begin{equation}
			\sigma \star: H^*(\M_0^\theta(v,w),\ZZ) \xrightarrow{\mathrm{top}} H^*(\M_{\sigma^*\lambda}^{\sigma^*\theta}(v,w),\ZZ)
			\xrightarrow{\Phi_\sigma^*}H^*(\M_{\lambda}^\theta(v,w),\ZZ)
			\xrightarrow{\mathrm{top}} H^*(\M_0^{\theta}(v,w),\ZZ) 
			.
		\end{equation}
		
	\end{definition}
	
	\begin{proposition}\label{proposition:si act on chern class}
		Let $\chi$ be a character of $G$ and $\sigma\in W(v,w)$. Then 
		\begin{equation}
			\sigma\star  \kappa(\chi) = \kappa(\sigma^* \chi). 
		\end{equation}
	\end{proposition}
	\begin{proof}		
		Under the hyper-K\"{a}hler setup, the tautological line bundle $\sheafO(\M_\lambda^\theta(v,w),\chi)$ is the restriction of the line bundle
		\begin{equation}\label{eq: global hyperkahler taut line bundle}
			\sheafO_{hk}(\chi):=\boldsymbol{\mu}\inverse (\mf{p}_{hk}^{\reg}) \times^U \C_{\chi|_U}
		\end{equation}
		to the fiber over $\boldsymbol{\zeta} = (i\theta,\lambda)$. 
		Here $\chi|_U$ denotes the restriction of $\chi$ to $U$. 
		The first Chern class $c_1(\sheafO_{hk}(\chi))$ gives a global section to the degree 2 part of 
		$\underline{\boldsymbol{\mu}}_* \underline{\ZZ}_{\boldsymbol{\mu}\inverse (\mf{p}_{hk}^{\reg}) / U} $. Therefore
		\begin{equation}
			\mathrm{top}:  c_1(\sheafO(\M_0^\theta(v,w),\chi))\mapsto   c_1(\sheafO(\M_0^{\sigma^*\theta}(v,w),\chi)). 
		\end{equation}
		Now the proposition follows from the definition of $\sigma\star$ and \Cref{prop: tautological line bundle and maffei}.
	\end{proof}
	\subsection{The group embedding $\iota$}\label{subsection: the group embedding iota}
	In this subsection we show that when $\tilde{v}\in\Sigma_0$, the map $\iota$ in \Cref{def: set theoretic iota} extends to a group embedding $W(v,w)_{\Ess} = W(v,w)$ into $\NamiWeyl(v,w)$ that intertwines the Maffei's action and the natural $\NamiWeyl(v,w)$-action on $H^2(\M_0^\theta(v,w),\C)$. 
	
	Suppose $\tilde{v}\in\Sigma_0$ and let $\L_i$ be a codimension 2 leaf of $\M_0^0(v,w)$. Its representation type takes the form $\tau = (\beta^1,1;\cdots;\alpha_{i_t},m_t)$ as in \eqref{eq: an isot decomp}. We assume the $\infty$-component of $\beta^1$ is 1. 
	Then $\hat \cartanh_i^*$ has a basis in bijection with the components of $\tau$ without $\beta_1$. 
	Recall $p_i: H^2(\M_0^\theta(v,w),\C) \to \cartanh_i$, see \Cref{Ivan's decomposition of cartan space,remark: projection from H2 to cartanh i}. 
	\begin{proposition}[{\cite[Proposition 3.4]{wu2023namikawa}}]\label{prop: degree of taut bundle to hihat}
		Let $\chi \in \ZZ^{Q_0}$. Then
		\begin{equation}
			p_i \kappa (\chi) = (\chi\cdot\beta^2,\chi\cdot \beta^3,\cdots,\chi\cdot \alpha_{i_t}) \in \cartanh_i\subset  \hat{\cartanh}^*_i.
		\end{equation}
	\end{proposition}
	Recall $Gen_{\Ess}$ is the set of generators of $W(v,w)_{\Ess}$. When $\tilde{v}\in\Sigma_0$, $Gen_{\Ess}$ is the same as the generating set of $W(v,w)$, \eqref{eq:gen of Wvw}. 
	\begin{theorem}\label{thm: Maffei and NW action of sj are same}
		Suppose $\tilde{v}\in \Sigma_0$, and pick $s_j\in Gen_{\Ess}$. 
		Then the following actions on $H^2(\M_0^\theta(v,w),\C)$ coincide: 
		\begin{enumerate}
			\item the Maffei's action $s_j\star$, and
			\item the action of $\iota(s_j)$ via the natural $\NamiWeyl(v,w)$-representation.
		\end{enumerate}
	\end{theorem}
	\begin{proof}
		Recall the decomposition \Cref{Ivan's decomposition of cartan space}. 
		Denote by $p_{\reg},p_i$ the projections from $H^2(\M_0^\theta(v,w),\C)$ to $H^2(\M_0^0(v,w)^{\reg},\C)$ and to $\cartanh_i$ respectively. Thanks to the surjectivity of $\kappa$, it suffices to show that for any $\chi\in \ZZ^{Q_0}$, 
		$p_i s_j\star \kappa(\chi) = p_i \iota(s_j) \kappa(\chi)$ for any $i$, and $p_{\reg} s_j\star \kappa(\chi) = p_{\reg} \iota(s_j) \kappa(\chi)$. 
		
		Let us consider the $p_i$ part. 
		Let $(\alpha_j,\tilde{v}) = 0$ so that $s_j\in W(v,w)$, and let $\chi \in \ZZ^{Q_0}$. Let $\tau = (\beta_1,1; \cdots;\beta_s,1; \alpha_{i_1},m_1;\cdots;\alpha_{i_t},m_t)$ be the representation type of the codimension 2 leaf $\L_i$. We compute
		\begin{equation}\label{eq:formula for sj action on pikappachi}
			\begin{split}
				p_i (s_j\star \kappa(\chi)) 
				&= p_i (\kappa(s_j^*\chi))
				\\ 
				&= (\beta_2 \cdot s_j^*\chi, \cdots, \alpha_{i_t}\cdot s_j^*\chi) \\
				& = (s_j\beta_2 \cdot \chi, \cdots, s_j\alpha_{i_t}\cdot \chi) \\
				&= (\beta_2 \cdot \chi - (\beta_2,\alpha_j)\chi_j , \cdots, \alpha_{i_t}\cdot \chi - (\alpha_{i_t},\alpha_j)\chi_j).
			\end{split}
		\end{equation}
		The first equality is \Cref{proposition:si act on chern class}, the second is \Cref{prop: degree of taut bundle to hihat}, and the latter 2 are obvious. 
		
		By \Cref{prop:a real root only appears in one leaf}, there is a unique codimension 2 leaf whose representation type $\tau$ contains $\alpha_j$. If $\L_i$ is not this leaf, then all $(\beta_k,\alpha_j)$ and $(\alpha_{i_\ell},\alpha_j)$ are 0. 
		In this case, $s_j$ acts trivially on $\cartanh_i$. 		
		
		On the other hand, suppose $\tau$ contains $\alpha_j$.  Without loss of generality we may assume $\alpha_j = \alpha_{i_1}$. 
		Consider the affine type quiver $\underline{Q}$ defined in \Cref{BS isotropic decomposition}. 
		Let $\underline{\alpha}_1,\cdots,\underline{\alpha}_s$ be the simple real roots of $\underline{Q}$ corresponding to $\beta_1,\cdots,\beta_s$ and $\underline{\alpha}_{i_1},\cdots,
		\underline{\alpha}_{i_t}$ the simple real roots of $\underline{Q}$ corresponding to $\alpha_{i_i},\cdots,\alpha_{i_t}$. 
		Remove the vertex associated to $\beta^1$ and we get a finite type quiver $\underline{Q}_{fin}$.  Let $\underline{s}_{?}$ be the reflection in the Weyl group of $\underline{Q}_{fin}$ corresponding to $\underline{\alpha}_?$, where $?$ can be $1,\cdots,s,i_1,\cdots,i_t$. If $\underline{\chi} \in \ZZ^{\underline{Q}_{fin}}$, then by \eqref{eq: formula of si act on chi}, 
		\begin{equation}\label{eq: sj in slice affine quiver action}
			\underline{s}_{i_1}^* \underline{\chi} =\underline{\chi} - \sum_{k=1}^s (\underline{\alpha}_k,\underline{\alpha}_{i_1}) \underline{\chi}_{i_1}  \underline{\alpha}_{k}^* - \sum_{\ell=1}^t (\underline{\alpha}_{i_\ell},\underline{\alpha}_{i_1}) \underline{\chi}_{i_1}  \underline{\alpha}_{i_\ell}^*
		\end{equation}
		Recall $(\underline{\alpha}_k,\underline{\alpha}_{i_1}) = (\beta_k,\alpha_{i_1})$ and $(\underline{\alpha}_{i_\ell},\underline{\alpha}_{i_1}) = (\alpha_{i_\ell},\alpha_{i_1})$, by (2) of \Cref{BS isotropic decomposition}. 
		Now set $\underline{\chi} = p_i\kappa(\chi)$. 
		Compare the last term of \eqref{eq:formula for sj action on pikappachi} and the right hand side of \eqref{eq: sj in slice affine quiver action}, and recall $\alpha_j = \alpha_{i_i}$, it is easy to see that they coincide. 
		It follows that $p_i( s_j\star \kappa(\chi) )= p_i (\iota(s_j) \kappa(\chi))$ for any $i$.
		
		Let us turn to the $p_{\reg}$ part. Since $\iota(s_j)$ acts trivially on $H^2(\M_0^0(v,w)^{\reg},\C)$, we need to show $s_j\star$ is trivial as well, which by \Cref{proposition:si act on chern class} is equivalent to that 		$p_{\reg}\kappa(\chi-s_j^*\chi) = 0$, and this is further equivalent to that the line bundle $\sheafO(\chi-s_j^*\chi,\M_0^\theta(v,w))$ is trivial when restricted to $\M_0^0(v,w)^{\reg}$. 
		
		The proof is very similar to that of \Cref{prop: tautological line bundle and maffei}.
		Write $\delta\chi = (\chi - s_j^*\chi)$, and observe 
		that $(\delta \chi)_k = \chi_j c_{jk}$. 
		We will prove the proposition by writing down a nowhere vanishing section of $\sheafO(\delta\chi ,\M_0^\theta(v,w))$ on $\M_0^0(v,w)^{\reg}$. 
		
		Suppose $r\in \mu\inverse(0)^{\theta-ss}$ and $[r]^\theta$ is its image in $\M_0^\theta(v,w)$. 
		Recall $\rho: \M_0^\theta(v,w) \to \M_0^0(v,w)$ is the natural projective map, and $\rho([r]^\theta) = [r^{ss}]^0$, where $r^{ss}$ is the semisimplification of the representation $r$ and $[r^{ss}]^0$ is the image of $r^{ss}$ in $\M_0^0(v,w)$. In particular, if $\rho([r]^\theta)  \in \M_0^0(v,w)^{\reg}$, then since $\tilde{v}\in \Sigma_0$, $r^{ss}$ has representation type $(\tilde{v},1)$. So $r$ is simple as a $Q$-representation. Therefore, the stabilizer of $r$ in $G$ is trivial, and the $G$-orbit of $r$ in $\mu\inverse(0)$ is closed.  
		It follows that $r$ is stable for any stability condition, in particular for both $\theta$ and $s_j^*\theta$. 	
		
		Recall the space $T_j$ defined in \eqref{eq:def of Tj} and the morphisms $A_j$ and $B_j$ in \eqref{eq: def of Aj and Bj}. By \cite[Lemma 32]{maffei2002remark}, the semistability for $\theta$ and $s_j^*\theta$ implies that $A_j$ is injective and $B_j$ is surjective. 	
		Fix an injection $A_{j,0} : V_j \to T_j$ and a surjection $B_{j,0}: T_j\to V_j $ that fit into a short exact sequence 
		$0\to V_j \xrightarrow{A_{j,0}} T_j \xrightarrow{B_{j,0}} V_j \to 0  $. For any $r\in \mu\inverse(0)^{\reg}$, there is some element $g\in \GL(T_j)$ such that $\im A_j = g \im A_{j,0}$, and the following commutative diagram defines automorphisms $\alpha_j$ and $\beta_j$ of $V$: 
		\begin{equation}
			\begin{tikzcd}
				0 \arrow[r] & V_j \arrow[r, "{A_{j,0}}"] \arrow[d, "\alpha_j"] & T_j \arrow[r, "{B_{j,0}}"] \arrow[d, "g"] & V_j \arrow[r] \arrow[d, "\beta_j"] & 0 \\
				0 \arrow[r] & V_j \arrow[r, "A_j"]                              & T_j \arrow[r, "B_j"]                      & V_j \arrow[r]                      & 0
			\end{tikzcd}.
		\end{equation}
		Define $F(r) =\det g\inverse\det \alpha_j \det \beta_j $. Then the map 
		\begin{equation}
			\mu\inverse(0)^{\reg} \to \mu\inverse(0)^{\reg} \times \C^*_{\delta\chi}, r\mapsto (r,F(r)^{\chi_j})
		\end{equation}
		is $G$-equivariant and descends to a nowhere vanishing section of $\sheafO(\delta\chi ,\M_0^\theta(v,w))$ on $\M_0^0(v,w)^{\reg}$. This finishes the proof of the theorem.  
	\end{proof}
	
	\begin{corollary}\label{cor: iota extends to group emb when v in Sigma0}
		Suppose $\tilde{v}\in\Sigma_0$. Then the map $\iota:Gen_{\Ess} \to \NamiWeyl(v_i) $ extends to a group embedding $W(v,w)\hookrightarrow \NamiWeyl(v,w)$. 
	\end{corollary}
	\begin{proof}
		The group $\NamiWeyl(v,w)$ is naturally identified with a subgroup of $\GL(\bigoplus_i \cartanh_i)$. By \Cref{thm: Maffei and NW action of sj are same}, the Maffei's action restricts to a representation of $W(v,w)$ in $\bigoplus_i \cartanh_i$ as well, and the images of the generators in $\GL(\bigoplus_i \cartanh_i)$ lie in $\NamiWeyl(v,w)$. Therefore, it suffices to show that the $W(v,w)$-representation in $\bigoplus_i \cartanh_i$ is faithful. 
		
		If $\sigma \in W(v,w)$ acts trivially on $\bigoplus_i \cartanh_i$ then $p_i (\sigma\star \kappa(\chi)) = p_i\kappa(\chi)$ for all $i$ and $\chi$. 
		By \Cref{proposition:si act on chern class}, this is equivalent to $p_i(\kappa (\sigma^*\chi - \chi)) = 0$ for all $i$ and $\chi$. 
		Assume $j\in Q_0$ and $(\tilde{v},\alpha_j) = 0$. By \Cref{prop: simple real cod 2 root}, the representation type of some codimension 2 leaf $\L_i$ contains $\alpha_j$.
		Then by \Cref{prop: degree of taut bundle to hihat}, $\alpha_j \cdot (\chi - \sigma^*\chi) = 0$ for any $\chi$. It follows that $\sigma \alpha_j = \alpha_j$ for all $j$ such that $(\tilde{v},\alpha_j) = 0$. We claim $\sigma = 1$, thus proving faithfulness. In fact, assume $\sigma \neq 1$, and let $\sigma = s_{i_1}s_{i_2}\cdots s_{i_k}$ be a reduced expression, where each $s_{i_j}$ is a generator of $W(v,w)$ in \Cref{prop:generators of Wvw}. 
		Then $\ell(ws_{i_k}) = \ell(w)-1$, so by \cite[Section 5.4]{humphreys1992reflection}, $\sigma \alpha_{i_k}<0$, contradicting that $\sigma \alpha_{i_k} = \alpha_{i_k}$.  
	\end{proof}
	\begin{corollary}
		If $\tilde{v} \in \Sigma_0$, then the group embedding $\iota$ intertwines the natural $\NamiWeyl(v,w)$-action and the Maffei's action of $W(v,w)$ on $H^2(\M_0^\theta(v,w),\C)$.
	\end{corollary}
	\begin{proof}
		This follows from \Cref{cor: iota extends to group emb when v in Sigma0} and \Cref{thm: Maffei and NW action of sj are same}.
	\end{proof}
	
	\subsection{More general dimension vectors}\label{sec: more general dimension vectors} 
	In this subsection we release the assumption $\tilde{v}\in \Sigma_0 $ and prove parts (2), (3) of \Cref{theorem:W vs NW} in full generality. 
	
	\begin{proof}[Proof of part (2) of \Cref{theorem:W vs NW}]
		It suffices to check that if $s_j$ lies in the set of generators \eqref{eq: generator of WvwIrr}, then $s_j\star$ is the identity automorphism. 
		
		So let $j$ be such a vertex. Let us assume $\theta >0$. 
		By the definition of $W(v,w)_{\Irr}$, inside $\overline{Q}$, the double quiver of $Q$, we have the following subquiver near $j$: 
		\begin{equation}    
			\begin{tikzpicture}[scale=0.9]
				\filldraw[black] (0,0) circle (1.5pt);
				\filldraw[black] (2,0) circle (1.5pt);
				\filldraw[black] (4,0) circle (1.5pt);
				\draw[->,thick] (.1,.1)--(1.9,0.1);
				\draw[->,thick] (1.9,-.1)--(.1,-0.1);
				\draw[->,thick] (2.1,.1)--(3.9,0.1);
				\draw[->,thick] (3.9,-.1)--(2.1,-0.1);
				\node at (-0.5,0) {$\cdots$};	
				\node at (4.5,0) {$\cdots$};
				\node at (0,-0.4) {$\alpha_{j-1}$};	
				\node at (2,-0.4) {$\alpha_{j}$};
				\node at (4,-0.4) {$\alpha_{j+1}$};
				\node at (0,0.4) {$1$};	
				\node at (2,0.4) {$1$};
				\node at (4,0.4) {$1$};
				\node at (1,-0.4) {$y_1$};
				\node at (1,0.4) {$x_1$};
				\node at (3,-0.4) {$y_2$};
				\node at (3,0.4) {$x_2$};
			\end{tikzpicture}
		\end{equation}
		where $x_1,x_2,y_1,y_2$ are the only arrows whose heads or tails are $j$, and the vertices of $Q$ to the left of $j-1$ and that to the right of $j+1$ are not connected to each other by any other arrows. 
		Without loss of generality, assume the extended vertex $\infty$ lies to the left of $j$. If $x_1 = 0$, then the sum of subspaces $V_k$, where $k$ lies to the left of $j$, is a proper subrepresentation of $\overline{Q}$ containing the image of the framing; if $x_2 = 0$, then the sum of subspaces $V_k$, where $k$ lies to the left of $j+1$, is such a proper subrepresentation of $\overline{Q}$. 
		Both contradict the $\theta$-semistability of $r$, by \cite[Theorem 10.32]{kirillov2016quiver}. Therefore, $x_1,x_2\neq 0$. 
		
		For any $\chi\in\ZZ^{Q_0}$, a character of $G$, write $\delta\chi = s_j^*\chi - \chi$. Then $(\delta\chi)_k = 0$ if $k\neq j$ and $j\pm 1$; $(\delta\chi)_j = 2\chi_j$, $(\delta\chi)_{j\pm 1} = -\chi_j$. Similarly to the proof of \Cref{thm: Maffei and NW action of sj are same}, the function $(x_1x_2\inverse)^{\chi_j}$ on $\mu\inverse(0)^{\theta-ss}$ descends to a nowhere vanishing global section of $\sheafO(\delta\chi,\M_0^\theta(v,w))$.
		It then follows from \Cref{proposition:si act on chern class} that the action $s_j\star$ is trivial. 
	\end{proof}
	
	Now let us consider the actions $s_j\star$ for $s_j\in Gen_{\Ess}$. We first generalize a half of \Cref{thm: Maffei and NW action of sj are same}. 
	\begin{proposition}\label{prop: sjstar and sj same action on cartanij under can decomp}
		Let $j\in\Supp v_i$ and $s_j\in Gen_{\Ess}$. Then the action $s_j\star$ and $\iota(s_j)$ coincide on the following subspace of $H^2(\M_0^\theta(v,w),\C)$ : 
		\begin{equation}
			\bigoplus_{\substack{v_i\in\Sigma_0\\ 1\le j \le  n_j}}\cartanh_{i,j} \oplus 
			\bigoplus_{\substack{v_i =m_i\delta_i \\m_i>1}} \cartanh_i 
			\oplus
			\bigoplus_{\substack{v_i =m_i\delta_i \\m_i>1}}\C.
		\end{equation}
		See \eqref{eq: ivans decomposition under canonical decomposition} for the notations.
		In particular, $s_j\star$ acts trivially on the last component. 
	\end{proposition}
	\begin{proof}
		Note that for any $\chi \in \ZZ^{Q_0}$, the characters $\chi$ and $s_j^*\chi$ are identical outside $\Supp v_i$. 
		Moreover, if $v_i = m_i\delta_i$, observe that for any $j\in \Supp \delta_i$ such that $s_j\in Gen_{\Ess}$, we have $s_j^*\chi \cdot \delta_i = \chi\cdot s_j \delta_i= \chi\cdot \delta_i$.  
		The proposition now follows from the classification of representation types in (a) of \Cref{cor: leaves rep types and NW under canonical decomp}, and the first half of the proof of \Cref{thm: Maffei and NW action of sj are same}. 
	\end{proof}

	Let us turn to the regular locus of $\M_0^0(v,w)$. The following result will be useful in the computation of $H^2(\M_0^0(v,w)^{\reg},\C)$ and is of independent interest.
	\begin{proposition}\label{prop:H1 of Xreg is zero}
		If $\pi: Y\to X$ is a conical symplectic resolution of conical symplectic singularities, then $H^1(X^{\reg},\C) = 0$.
	\end{proposition}
	\begin{proof}
		Suppose $X,Y$ have complex dimension $d$. 
		Let $E\subset Y$ be the exceptional locus, so that $Y\setminus E \cong X^{\reg}$. The long exact sequence of Borel-Moore homology for $Y,E$ and $Y\setminus E$ implies 
		\begin{equation}
			\cdots  H^1(Y,\C) \to H^1(X^{\reg},\C) \to H_{2d-2}(E,\C) \to  H^2(Y,\C) \to H^2(X^{\reg},\C) \cdots
		\end{equation}
		where we have used the Poincare duality and the fact that $Y,X^{\reg}$ are smooth. 
		By \cite[Proposition 2.5]{braden2016conical}, the odd cohomology spaces of a conical symplectic resolution vanish, so $H^1(Y,\C)=0$. 	
		By \Cref{Ivan's decomposition of cartan space}, we see $H^2(Y,\C) \to H^2(X^{\reg},\C)$ is surjective and the kernel has dimension equal to $\sum \dim \cartanh_i$. 
		By \cite[Section 4.1]{namikawa2011poisson}, if $\L_i$ is a codimension 2 symplectic leaf of $X$, then the number of (codimension 1 in $Y$) irreducible components of $\pi\inverse(\L_i)$ is precisely $\dim\cartanh_i$. 	
		By \cite[Lemma 2.11]{kaledin2006symplectic}, $\pi$ is semismall, so the preimages under $\pi$ of the symplectic leaves of $X$ with codimension $\ge 4$ cannot contribute to $H_{2d-2}(E,\C)$. 
		
		Therefore, $\dim H_{2d-2}(E,\C) = \sum\dim \cartanh_i$, which forces $H_{2d-2}(E,\C) \to  H^2(Y,\C)$ to be an injection. It follows that $H^1(X^{\reg},\C) = 0$. 
	\end{proof}
	\begin{corollary}\label{cor: Mreg has no H1}
		If $v_i\in\Sigma_0$, then $H^1(\M_0^0(v_i)^{\reg},\C) = 0$.
	\end{corollary}
	\begin{proof}
		By the construction of the canonical decomposition, each $v_i$ has some component $(v_i)_k $ equal to $ 1$. Therefore, by \Cite[Theorem 1.5]{bellamy2021symplectic} and the subsequent remarks,
		$\M_0^\theta(v_i)\to \M_0^0(v_i)$ is a conical symplectic resolution for generic $\theta$. Now use \Cref{prop:H1 of Xreg is zero}. 
	\end{proof}
	\begin{proposition}\label{prop: quotient singularity has no H1 and H2}
		If $\delta$ is the minimal imaginary root supported on an affine type quiver, $m\ge 2$, then
		\begin{equation}
			H^1(\M_0^0(m\delta)^{\reg},\C) = H^2(\M_0^0(m\delta)^{\reg},\C) = 0.
		\end{equation} 
	\end{proposition}
	\begin{proof}
		It is well known that $\M_0^0(m\delta) = \C^{2m}/\Gamma_m$, where $\Gamma_m = \Gamma^{\times m} \rtimes S_m$, $\Gamma$ is the finite subgroup of $\SL(2,\C)$ corresponding to the type of $\delta$. We can thus identify $\M_0^0(m\delta)^{\reg} = (\C^{2m} \setminus Z) / \Gamma_m$, where $Z$ is a union of complex codimension 2 subspaces of $\C^{2m}$, and $\Gamma_m$ acts freely on $\C^{2m} \setminus Z$. The proposition follows from the long exact sequence in Borel-Moore homology for $\C^{2m}, Z, \C^{2m}\setminus Z$ and dimension reasons.  
	\end{proof}
	\begin{corollary}\label{cor: H2reg is sum of H2reg of sigma0 parts}
		Let $\M_0^0(v,w) = \prod_{i=1}^k \M_0^0(v_i)$ be the canonical decomposition of $\M_0^0(v,w)$ in \Cref{prop: decompose tildev into a string}. Then 
		\begin{equation}
			H^2(\M_0^0(v,w)^{\reg},\C) = \bigoplus_{v_i\in\Sigma_0} H^2(\M_0^0(v_i)^{\reg},\C).
		\end{equation}
	\end{corollary}
	\begin{proof}
		This follows from \Cref{cor: Mreg has no H1}, \Cref{prop: quotient singularity has no H1 and H2} and the Kunneth formula. 
	\end{proof}
	\begin{proposition}\label{prop: sj ess trivial action on H2reg}
		If $s_j\in Gen_{\Ess}$, then $s_j\star$ acts trivially on $H^2(\M_0^0(v,w)^{\reg},\C) $. 
	\end{proposition}
	\begin{proof}
		Let $\chi\in\ZZ^{Q_0}$ and $\delta\chi = s_j^*\chi - \chi$.  
		Thanks to \Cref{cor: H2reg is sum of H2reg of sigma0 parts}, it suffices to show that the restriction of $\sheafO(\delta\chi,\M_0^\theta(v,w))$ to each $\M_0^0(v_i)^{\reg}$ is trivial, where $v_i\in\Sigma_0$. Recall for $r\in\mu\inverse(0)^{\theta-ss}$, we denote $[r]^\theta$ to be its image in $\M_0^\theta(v,w)$. Assume $\rho([r]^\theta) \in \M_0^0(v,w)^{\reg}$. 
		
		\textbf{Claim.} The restriction of $r$ to any $\Supp v_i$ is simple as a representation of $Q|_{\Supp v_i}$. 
		
		Once the claim is proved, we can apply the constructions in the second half of the proof of \Cref{thm: Maffei and NW action of sj are same} to produce a nowhere vanishing section of $\sheafO(\delta\chi,\M_0^\theta(v,w))$ on $\M_0^0(v_i)^{\reg}$, and the proposition is proved. 
		
		Let us prove the claim. Since $\rho([r]^\theta) = [r^{ss}]^0$, there is a filtration of the $\overline{Q}$-representation $r$, part of which is 
		\begin{equation}
			\cdots \subset R_n \subset R_{n+1} \subset\cdots 
		\end{equation}
		such that $R_{n+1}/R_n$ is a simple representation of $\overline{Q}$ with dimension $v_i$. To prove the claim, it suffices to show $R_{n+1}/R_n$ is also simple as a representation of $\overline{Q}|_{\Supp v_i}$. 
		
		Suppose not; then we have a chain $R_n\subsetneq S \subsetneq R_{n+1}$ of $\overline{Q}|_{\Supp v_i}$ representations.
		Let $i_{conn}$ be any vertex $\Supp v_i$ that connects to a vertex outside $\Supp v_i$. By the construction, $V_{i_{conn}}$ is one-dimensional. 
		If $S$ does not contain any such  $V_{i_{conn}}$, then $S$ is already a $Q$-representation, contradiction. 
		Suppose now $S$ contains such a $V_{i_{conn}}$. Since $\dim V_{i_{conn}} = 1$, the subspace of $R_{n+1}$ generated by $S$ under the $\overline{Q}$-action is the same as that as a $\overline{Q}|_{\Supp v_i}$-action, and $S$ is again a $\overline{Q}$-representation, contradiction. The claim, and hence the proposition, are proved. 
	\end{proof}
	\begin{proof}[Proof of (3) of \Cref{theorem:W vs NW}]
		This part of the theorem now follows from \Cref{prop: sjstar and sj same action on cartanij under can decomp} and \Cref{prop: sj ess trivial action on H2reg}. 
	\end{proof}
	
	\section{Coincidence of the two actions}\label{wrap up part 4}
	Let $\sigma\in W(v,w)$ and set $\gamma = (1\times\iota)(\sigma)\in \NamiWeyl(v,w)$. Recall $\mathfrak{p} = \C^{Q_0}$ and $B_Y = H^2(\M_0^\theta(v,w),\C)$.  We have the following commutative diagrams whose vertical arrows are isomorphisms: 
	\begin{equation}
		\begin{tikzcd}
			\mathfrak{p} \arrow[r, "\kappa"] \arrow[d, "\sigma"] & B_Y \arrow[d, "\gamma"] \\ 
			\mathfrak{p} \arrow[r, "\kappa"]  & B_Y
		\end{tikzcd} \hspace{0.5in}
		\begin{tikzcd}
			\mathfrak{p^\circ} \arrow[r, "\kappa"] \arrow[d, "\sigma"] & B_Y^\circ \arrow[d, "\gamma"] \\ 
			\mathfrak{p^\circ} \arrow[r, "\kappa"]  & B_Y^\circ
		\end{tikzcd}
	\end{equation}
	where the left diagram follows from part (3) of \Cref{theorem:W vs NW}, and the right diagram follows from the left by restriction. 
	\begin{proof}[Proof of (4) of \Cref{theorem:W vs NW}]
		Let $\lambda\in\mathfrak{p}^\circ$ be generic, and consider the following diagram. 
		\begin{equation}
			\begin{tikzcd}[sep=scriptsize]
				\M_{\sigma^*\lambda}^\theta(v,w) \arrow[r,equal]& {\M_{\sigma^*\lambda}^{\sigma^*\theta}(v,w)} && {\M^{\sigma^*\theta}_{\mathfrak{p}^{\circ}} (v,w)} && {\mathcal{Y}^{\circ}} \\
				{\M_{\lambda}^\theta(v,w)} && {\M^{ \theta}_{\mathfrak{p}^{\circ} }(v,w)} && {\mathcal{Y}^{\circ}} \\
				& {\sigma^*\lambda} && {\mathfrak{p}^{\circ}} && {B_Y^{\circ}} \\
				\lambda && {\mathfrak{p}^{\circ}} && {B_Y^{\circ}}
				\arrow["\kappa", from=4-3, to=4-5]
				\arrow["{\gamma}"', from=4-5, to=3-6]
				\arrow["\kappa"{pos=0.3}, from=3-4, to=3-6]
				\arrow["\sigma^*", from=4-3, to=3-4]
				\arrow["\mu"'{pos=0.6}, from=2-3, to=4-3]
				\arrow["{{ }}", from=1-4, to=3-4]
				\arrow["{\Psi_{\gamma}}"{pos=0.3},"\sim"', from=2-5, to=1-6]
				\arrow["K"{pos=0.4}, from=2-3, to=2-5]
				\arrow["{m_Y}"{pos=0.3}, from=2-5, to=4-5]
				\arrow["{m_Y}", from=1-6, to=3-6]
				\arrow["{K'}", dashed, from=1-4, to=1-6]
				\arrow["{{\Phi_\sigma}}","\sim"', to=1-4, from=2-3]
				\arrow[hook, from=2-1, to=2-3]
				\arrow["\mu"', from=2-1, to=4-1]
				\arrow[hook, from=4-1, to=4-3]
				\arrow[from=2-1, to=1-2]
				\arrow[hook, from=1-2, to=1-4]
				\arrow[from=1-2, to=3-2]
				\arrow[from=4-1, to=3-2]
				\arrow[hook, from=3-2, to=3-4]
			\end{tikzcd}
		\end{equation}
		The right and middle faces are Cartesian by their constructions. The dashed arrow $K'$ is defined to be the composition $\Psi_\gamma \circ K \circ \Phi_\sigma\inverse$, so the right cube commutes. Note that all arrows in the $\nearrow$ direction are isomorphisms. It follows that all the faces in the right cube involving $K'$ are also Cartesian. 
		The left cube is obtained by restricting the middle face of the diagram to $\lambda \mapsto \sigma^*\lambda$. It follows that the whole diagram commutes. 
		
		In particular, the Maffei's isomorphism
		\begin{equation}
			\Phi_\sigma: \M_\lambda^\theta(v,w) \to \M_{\sigma^*\lambda}^{\sigma^*\theta}(v,w) = \M_{\sigma^*\lambda}^{\theta}(v,w)
		\end{equation} 
		and the restriction of $\Psi_\gamma$ to $m_Y\inverse(\kappa(\lambda))$: 
		\begin{equation}
			\M_\lambda^\theta(v,w)  = m_Y\inverse(\kappa(\lambda)) \xrightarrow{\Psi_\gamma} m_Y\inverse(\gamma\kappa(\lambda)) = m_Y\inverse(\kappa(\sigma^*\lambda)) = \M_{\sigma^*\lambda}^\theta(v,w)
		\end{equation} coincide. 
		By the definitions of $\gamma*$ in \Cref{def: symp springer action} and $\sigma\star$ in \Cref{def: maffei action}, 
		we conclude that $\sigma\star = \gamma*$ as automorphisms of $H^*(\M_0^\theta(v,w),\C)$. This finishes the proof of (4) of \Cref{theorem:W vs NW}. 
	\end{proof}
	\begin{remarklabeled}
	 We have partially recovered \Cref{remark springer equals natural on H2}. Namely, $\gamma*$ coincides with the natural  $\gamma$-action on $H^2(\M_0^\theta(v,w),\C)$ from \Cref{remark natural action NW on namikawacartan}, when $\gamma$ takes the form $(1\times \iota)(\sigma)$ for some $\sigma\in W(v,w)$. 
	\end{remarklabeled}

	\printbibliography

\end{document}